\documentclass[12pt,a4paper]{article}
\usepackage{amsmath,amsfonts,amsthm,amssymb,mathrsfs}
\usepackage[mathscr]{eucal}
\usepackage{graphicx}
\usepackage{algorithm,algorithmic}
\usepackage{times}
\usepackage{authblk}
\usepackage[usenames,dvipsnames]{color}
\numberwithin{equation}{section}

\theoremstyle{plain}
\newtheorem{thm}{Theorem}[section]
\newtheorem{lem}[thm]{Lemma}

\newtheorem*{lem*}{Lemma}
\newtheorem*{prop*}{Proposition}
\newtheorem*{thm*}{Theorem}
\newtheorem*{clm1*}{Claim 1}
\newtheorem*{clm2*}{Claim 2}
\newtheorem*{clm3*}{Claim 3}
\newcommand{\be}{\begin{equation}}
\newcommand{\ee}{\end{equation}}
\newcommand{\bear}{\begin{eqnarray*}}
\newcommand{\eear}{\end{eqnarray*}}
\newcommand{\bearr}{\begin{eqnarray}}
\newcommand{\eearr}{\end{eqnarray}}
\newcommand{\bi}{\begin{itemize}}
\newcommand{\ei}{\end{itemize}}
\newcommand{\bn}{\begin{enumerate}}
\newcommand{\en}{\end{enumerate}}
\newcommand{\I}{\item}

\newcommand{\grad}{{\nabla}}

\newcommand{\ep}{\epsilon }

\newcommand{\R}{{\mathbb {R}}}

\newcommand{\p}{{\partial }}

\theoremstyle{definition}

\theoremstyle{remark}
\newtheorem{remark}{Remark}[section]

\newcommand{\na}{\nabla}

\newcommand{\lt}{\left}
\newcommand{\rt}{\right}

\newcommand{\mf}{\mathbf}

\newcommand{\ml}{\mathcal}
\newcommand{\ms}{\mathscr}

\newcommand{\bs}{\boldsymbol}

\newcommand{\tq}{{q}}
\newcommand{\bq}{\bar{q}}
\begin{document}

\fontsize{12}{18pt plus.4pt minus.3pt}\selectfont

\title{Ground State Patterns and Phase Transitions of Spin-1
Bose-Einstein Condensates via $\Gamma$-Convergence Theory
}
\author[,1,3]{I-Liang Chern\footnote{chern@math.ntu.edu.tw}}
\author[,2]{Chiu-Fen Chou\footnote{chouchiufen@gmail.com}}
\author[,1]{Tien-Tsan Shieh\footnote{ttshieh@ncts.ntu.edu.tw}}
\affil[1]{National Center for Theoretical Science, National Taiwan University, Taipei, 106, Taiwan}
\affil[2]{Department of Applied Mathematics, National Chiao Tung University, Hsinchu 300, Taiwan}
\affil[3]{Institute of Applied Mathematical Science, National Taiwan University, Taipei, 106, Taiwan}
\date{}
\maketitle
\begin{abstract}
 We develop an analytic theory  for the ground state patterns and their phase transitions  for  spin-1 Bose-Einstein condensates  on a bounded domain in the presence of  a uniform magnetic  field.  Within the Thomas-Fermi approximation, these ground state patterns are composed of four basic states: magnetic state, nematic state, two-component state and three-component state, separated by interfaces.  A complete phase diagram of the ground state patterns are found analytically with  different quadratic Zeeman energy $q$ and total magnetization $M$  for both ferromagnetic and antiferromagnetic systems.  Using the $\Gamma$-convergence technique, it is found that the  semi-classical limits of these ground states minimize an
 energy functional which consists of interior interface energy plus a boundary contact energy.
As a consequence,  the interface between two different basic states has constant mean curvature, and the contact angle between the interface and the boundary obeys  Young's relation.
\end{abstract}

%
%
%
%
%
%
%
%
%
%
%
%
%
%
%
%
%
%
%
%

\section{Introduction}
\subsection{The Gross-Pitaevskii equation for spin-1 BECs}
In the 1920s, Bose and Einstein~\cite{Bose1924,Einstein1925}  predicted a new state of matter.  At very low temperature, very dilute Boson gases such as alkai gases tend to occupy their lowest quantum state.  The de Broglie wave of the bosons  is coherent and their lowest quantum state becomes apparent, called  Bose-Einstein condensation (BEC). It was only recently, in 1995,  BEC was  first realized in laboratories by two groups independently, Cornell and  Wieman as well as Ketterle~\cite{Anderson95,Davis95}.  Through the mean field approximation and Hartree's ansatz, the mean-field state of an $N$-particle boson gases  can be described by a complex order parameter $\psi(x,t)$. Its dynamics is modelled by the Gross-Pitaevskii equation~\cite{Dalfovo99,Gross61,Pitaevskii61}:
\[
i \hbar \p_t \psi = \frac{\delta \ml{E}[\psi]}{\delta \psi^*},  \qquad\textrm{where } \ml{E}[\psi] = \int H(\psi)\, dx,
\]
and the Hamiltonian of the system is written as
\[
H(\psi) :=  \frac{\hbar^2}{2m}|\grad \psi|^2 + V(x) |\psi|^2 + \frac{\beta}{2}|\psi|^4.
\]
Here, $\psi^*$ denotes the complex conjugate of $\psi$,  and $\delta \ml{E}/\delta \psi^*$ is the variation of $\ml{E}$ with respect to $\psi^*$.  The function $V(x)$  is the trap potential satisfying $V(x) \to \infty$ as $|x|\to \infty$. The parameter $m$ is the mass of the boson, and $\beta$ is the product of the particle number $N$ and interaction strength.   This interaction  is attractive when $\beta < 0$ and repulsive when $\beta > 0$.
This mean field  model was rigorously  justified as a limit of the $N$-particle quantum system by Lieb et al.~\cite{Lieb2001,Lieb2001-2} for the ground state cases and by Erd\"os et al.~for the dynamic cases~\cite{Yau2010}.

When an optical trap is applied to confine BECs,
all their hyperfine spin states are active. Such a BEC with an internal spin freedom is called a spinor BEC~\cite{ho1998,Ohmi98,Law98}  and   was realized in laboratories with $^{23}$Na atoms in 1998~\cite{stenger1998,stamper1998,isoshima1999,bloch2008}.
In the mean field theory, a physical state of a spin-$f$ BEC is described by  $(2f+1)$-components of complex order parameters $\Psi = (\psi_{f},\psi_{f-1},\dots,\psi_0\dots,\psi_{-f})^T$ and its dynamics is governed by a multi-component Gross-Pitaeviskii equation~\cite{ho1998,Ohmi98}.  In this paper, we concern ourself with a quantum system of the  spin-1 BEC,  whose dynamics are described by a generalized Gross-Pitaevskii equation
\be
i \hbar \p_t \Psi = \frac{\delta \ml{E}[\Psi]}{\delta \Psi^\dagger}, \qquad \ml{E}[\Psi] = \int H(\Psi)\, dx,
\ee
where $\Psi = (\psi_1,\psi_0,\psi_{-1})^T$  and the Hamiltonian is given in the form
\begin{align*}
H(\Psi) &:=  H_{kin}(\Psi)+H_{pot}(\Psi)+H_{n}(\Psi)+H_s(\Psi) \\
&:= \frac{\hbar^2}{2m}|\grad \Psi|^2 + V(x) |\Psi|^2 + \frac{c_n}{2}|\Psi|^4 + \frac{c_s}{2} |\Psi^\dagger \mathbf{F} \Psi|^2.
\end{align*}
Here, $\Psi^\dagger := (\psi_1^*,\psi_0^*,\psi_{-1}^*)$ and $\mathbf{F}=(F_x,F_y,F_z)$ is the spin-1 Pauli operator:
\[
F_x = \frac{1}{\sqrt{2}}\lt(\begin{array}{ccc}0&1&0\\1&0&1\\0&1&0\end{array}\rt),
F_y = \frac{i}{\sqrt{2}}\lt(\begin{array}{ccc}0&-1&0\\1&0&-1\\0&1&0\end{array}\rt),
F_z = \lt(\begin{array}{ccc}1&0&0\\0&0&0\\0&0&-1\end{array}\rt).
\]
The term $\frac{c_n}{2}|\Psi|^4$ represents  the spin-independent interaction between bosons, whereas $\frac{c_s}{2}|\Psi^\dagger \mathbf{F}\Psi|^2$ stands for  the spin-exchange interaction between bosons.  The spin-independent interaction is attractive if $c_n < 0$ and repulsive if $c_n > 0$.
This BEC system is called ferromagnetic if $c_s <0$ and antiferromagnetic  if $c_s > 0$.  In experiments, an example of a ferromagnetic system is the
alkali atom $^{87}$Rb with $c_n\approx 7.793$, $c_s\approx -0.0361$~\cite{Anderson95},   whereas for the
alkali atom $^{23}$Na, $c_n \approx 15.587$ and $c_s\approx  0.4871$~\cite{Davis95}, and it is an antiferromagnetic system.

We are interested in the ground state patterns of spin-1 BECs in the presence of an external  uniform magnetic field.  The interaction of atoms  with the applied magnetic field, say $B \hat{z}$, introduces an additional energy, called the Zeeman energy:
\[
\ml{E}_{Zee} [\Psi]= \int H_{Zee}(\Psi)\,dx, \qquad H_{Zee}(\Psi) := \sum_{j=-1}^1 E_j(B) |\psi_j|^2,
\]
where  $E_j(B)$ is the Zeeman energy shift for each component under the magnetic field $B\hat{z}$. It is convenient for later discussion to write the Zeeman energy $H_{Zee}$ as
\[
H_{Zee} = \sum_{j=-1}^1 E_j(B)  |\psi_{-1}|^2
= q (|\psi_1|^2 + |\psi_{-1}|^2) + p (|\psi_1|^2 -  |\psi_{-1}|^2) + E_0 (|\psi_1|^2+|\psi_{0}|^2|+\psi_{-1}|^2),
\]
where
\[
p = \frac{1}{2} (E_{-1}-E_1), \
q = \frac{1}{2} (E_{-1}+E_1-2E_0).
\]
The parameters $p$ and $q$ are called linear and quadratic Zeeman energy, respectively, in the physics literature.

Notice that the Gross-Pitaevskii equation for spin-1 BECs possess two invariants: the total mass
\be \label{constraintN}
\ml{N}[\Psi] := \int |\Psi|^2 \, dx = \int |\psi_1|^2 + |\psi_0|^2 + |\psi_{-1}|^2\, dx,
\ee
and the total magnetization
\be \label{constraintM}
\ml{M}[\Psi] := \int |\psi_1|^2 - |\psi_{-1}|^2\, dx,
\ee
which can be derived by direct calculation.
The ground states of the system are those physical states of lowest  energy, given  fixed total mass and fixed total magnetization.  This defines a variational problem for ground states in the presence of a uniform magnetic field:
\begin{equation}
\min_{\Psi}\left\{ \ml{E}[\Psi]:\,   \ml{N}[\Psi]=N, \ml{M}[\Psi] = M\right\} \label{eq:1.4}.
\end{equation}
where the energy takes the form
\begin{align*}
\ml{E}[\Psi] & = \int \lt(H_{kin}+H_{pot}+H_n+H_s+ {q(|\psi_1|^2+|\psi_{-1}|^2)}\rt)\, dx + E_0N + pM\\
   & =\int \frac{\hbar^2}{2m}|\grad \Psi|^2 + V(x) |\Psi|^2 + \frac{c_n}{2}|\Psi|^4 + \frac{c_s}{2} |\Psi^\dagger \mathbf{F} \Psi|^2 + q (|\psi_1|^2+|\psi_{-1}|^2)\,dx + E_0N + pM.
\end{align*}
It is  observed that the parameter $p$ plays no role in minimization  for  fixed total magnetization $M$.
Thus, our goal is to study the ground state patterns and their phase transitions in the parameter plane $(q,M)$.

\subsection{A brief  survey of ground state problems}
There are several studies concerned with ground states of spin-1 BEC systems.
The existence of ground states for spin-1 BECs in three dimensions with $|c_s|< c_n$ was given in ~\cite{LinChern2013}.  The non-existence result in three dimensions for $c_n < 0$ was given in ~\cite{BaoCai2012}.  Other existence/non-existence results in one dimension are given in ~\cite{Cao2011}.

In the case of no applied  magnetic field, it is proven that the ground state is the so-called single mode approximation (SMA) for a ferromagnetic system.  That is, the ground state has the form $(\gamma_1,\gamma_0,\gamma_{-1})\psi$, where $\gamma_i\in\mathbb R^+$, and  $\psi$ is a scalar field.
On the other hand, the ground state is a so-called  two-component state  $(2C)$ for antiferromagnetic systems.   The above results are given numerically in ~\cite{BaoLim2008,ChenChernWang2011}, and analytically in ~\cite{LinChern2011}.

When there is an applied magnetic field, it is found that there is a phase transition from 2C to 3C, i.e. all three components are not zeros , as $q> q_{cr}>0$ for antiferromagnetic systems.  This phase transition phenomenon was also observed in the laboratory~\cite{stenger1998,Jacob2012}.  It was found numerically in ~\cite{Zhang-Yi-You2003,BaoLim2008,CCW2014} and proven analytically in ~\cite{LinChern2013}.

Due to the fact that the coefficient $\frac{\hbar^2}{2m}\ll 1$, people have paid attention to the Thomas-Fermi approximation of \eqref{eq:1.4} which simply ignores its kinetic energy. Under such an approximation, existence results for ground states and phase transitions  with $|c_s| << c_n$ have also been studied by many physicists, see~\cite{Zhang-Yi-You2003,matuszewski2008,matuszewski2009,matuszewski2010}.

Without ignoring the kinetic energy, the problem \eqref{eq:1.4} is a variational problem of singular perturbation type which is similar to the Van der Waals-Cahn-Hilliard gradient theory of phase transition for a fluid system of binary phases confined to a bounded domain under  isothermal conditions. Gurtin \cite{G2} had conjectured about the asymptotic behavior of the variational model
$$ \inf\left\{ \int_\Omega \left( \epsilon |\nabla u|^2 + \frac{1}{\epsilon}W(u)\right):\, u\in W^{1,2}(\Omega),\int_\Omega u(x)\,dx=m \right\},$$
where $W:\mathbb R^n\to\mathbb R^+$ is a double-well potential and $m$ is the total mass of the fluid. The problem have been studied extensively by many authors through de Giorgi's $\Gamma$-convergence theory. The scalar case ($n=1$) is studied by \cite{M,KS,S}. In \cite{FT,S,S1}, the vectorial case ($n>1$) is studied. With a given Dirichlet boundary condition, a sharp interface limit of the energy functional is considered in \cite{I2,ORS,Shieh}.
 A minimal interface problem arising from a two-component BEC in the regime of strong coupling and strong segregation was studied in \cite{AR} via $\Gamma$-convergence. They have formulated the problem in term of total density and spin functions, which convert the energy into a sum of two weighted Cahn-Hilliard energy.

\subsection{Contribution of this paper}
This paper considers ground state patterns and their phase transition diagram on the
$q$-$M$ plane for the case $|c_s| < c_n$.  The contribution of this paper includes
\begin{itemize}
\item
we find all possible ground states configurations for the spin-1 BEC system in its Thomas-Fermi approximation and give a complete phase diagram on the parameter space $(q,M)$;
\item a sharp interface limit of the BEC system is derived through de Gorgi's $\Gamma$-convergence.
\end{itemize}
This paper is organized as the follows. Section 2 is a reformulation and simplification of the problem.  Section 3 and 4 are devoted to the Thomas-Fermi approximations and the $\Gamma$-convergence results.

\section{Formulation of the problem}
In this Section, the variational problem \eqref{eq:1.4} for ground states of the spin-1 BEC  system is formulated into a real-valued variational problem with dimensionless coefficients.
\paragraph{Dimensional Analysis}
We minimize the energy function $\ml{E}[\Psi]$ under the two constraints (\ref{constraintN}), (\ref{constraintM}).  The corresponding Euler-Lagrange equations read
\bear
(\mu+\lambda)\psi_1 &=& \lt[-\frac{\hbar^2}{2m}\grad^2 + V(x)+  q+c_0 n\rt]\psi_1 +c_2 (n_1+n_0-n_{-1})\psi_1 + c_2 \psi_{-1}^*{\psi_0}^2 \\
\mu\psi_0 &=& \lt[-\frac{\hbar^2}{2m}\grad^2 + V(x)+c_0 n\rt]\psi_0 +2 c_2 (n_1-n_{-1})\psi_0 + c_2 \psi_{1}
\psi_{-1}{\psi_0^*}\\
(\mu-\lambda)\psi_{-1} &=& \lt[-\frac{\hbar^2}{2m}\grad^2 + V(x)+ q+c_0 n\rt]\psi_{-1} +c_2 (n_{-1}+n_0-n_{1})\psi_{-1} + c_2 \psi_{1}^*\psi_0^2.
\eear
Here,  $\mu$ and $\lambda$ are the two Lagrange multipliers corresponding to the two constraints.
We perform rescaling: $x=Lx'$, $\psi = L^{-3/2} \psi'$, then compare the dimensions of the first equation:
\[
 \frac{\hbar^2}{2m} L^{-3/2}L^{-2}{\grad'}^2\psi' + V L^{-3/2}\psi'+ q L^{-3/2}\psi' + c_n L^{-3}n'\psi'_1 + c_s L^{-3} .
 \]
From this, we define new parameters:
\begin{align*}
\ep^2 &= \frac{\hbar^2}{2m}L^{-7/2} \\
V' &= VL^{-3/2} \\
q' &= qL^{-3/2} \\
c_n' &= c_n L^{-3}  \\
c_s' &= c_s L^{-3}.
\end{align*}
Dropping the primes, we get the Hamiltonian $H_\ep$ defined in (\ref{eq:Hamilton_ep}) below.
%
\paragraph{Bounded domain problem with zero potential}
The trap potential $V$ in the laboratory satisfies $V(x) \to \infty $ as $|x|\to \infty$, which leads to the exponential decay of the ground states.  This fact can be derived from standard elliptic PDE theory~\cite{Lieb2001, LinChern2013}.   In particular, the quadratic potential (which is commonly  used in laboratory)  is  close to zero potential near the trapped center.
Thus, it is reasonable to consider the following potential with infinite well
\[
V(x) = \lt\{\begin{array}{ll} 0 & \mbox{ if } x \in \Omega\\ \infty & \mbox{ if } x\ne \Omega,\end{array}\right.
\]
where $\Omega$ is a smooth bounded domain in $\R^3$.  This corresponds to the constrained variational problem in a bounded domain with zero Dirichlet boundary condition.
\paragraph{Reduction to a real-valued problem}
To study the ground states, we also notice that we can limit ourselves to those order parameters $\psi_j$ with constant phases.  In fact, if we express $\psi_j = u_j e^{i\theta_j}$, $u_j \ge 0$, $j = 1, 0, -1$, then the kinetic energy is $|\grad \psi_j |^2 = |\grad u_j |^2 + u_j^2 |\grad \theta_j|^2$, which has minimal energy when $\grad \theta_j \equiv 0$.   In this situation, the only term in the Hamiltonian $H$ which involves phases is
\[
H_s = \frac{c_s}{2} \left((u_1^2-u_{-1}^2)^2 + 2 u_0^2(u_1^2+u_{-1}^2 + 2u_1 u_{-1}\cos(\Delta \theta)\right)
\]
where
$\Delta\theta = \theta_1+\theta_{-1}-2\theta_0$.
The Hamiltonian $H_s$ has a minimal value when
\[
\Delta \theta = \left\{ \begin{array}{ll}
0 & \mbox{ if }c_s < 0,\\
\pi & \mbox{ if }c_s > 0.
\end{array}\right.
\]
The resulting $H_s$ becomes
\be
H_s(\mf{u}) = \frac{c_s}{2} \left((u_1^2-u_{-1}^2)^2 + 2 u_0^2(u_1-\mbox{sgn}(c_s)u_{-1})^2\right).
\ee

\paragraph{Summary of the Problem}
To summarize the above simplifications, we shall consider the following constrained variational problem:
\[
(P_\ep)\quad \left\{ \begin{array}{l}
\inf \lt\{ \ml{E}_\ep[\mf{u}] := \int_\Omega H_\ep(\mf{u})\, dx \ \mid \ \mf{u}\in (H^1_0(\Omega,\R_+))^3 \cap L^4(\Omega,\R_+^3)\rt\}\\
\mbox{ subject to the constraints }\\
\ml{N}[\mf{u}]:= \int_{\Omega} (u_1^2+u_0^2+u_{-1}^2) \, dx = N,\
\ml{M}[\mf{u}]:=\int_{\Omega} (u_1^2-u_{-1}^2) \, dx = M.
\end{array}\right.
\]
Here, the Hamiltonian $H_\ep(\mf{u})$ is
\be \label{eq:Hamilton_ep}
H_\ep(\mf{u}) := \ep^2 \sum_{j=-1}^1|\na u_j|^2 +  \frac{c_n}{2} |\mf{u}|^4 + \frac{c_s}{2}\big[2u_0^2(u_1-\textrm{sgn}(c_s)\,u_{-1})^2+(u_1^2-u_{-1}^2)^2\big]+
 q(u_1^2+u_{-1}^2).
\ee

%
%
%
%
%
%
\section{The ground states in Thomas-Fermi approximation}\label{sec:TF}
The semi-classical regime is the case where $\ep$ is small (for instance,  choosing large $L$).  The Thomas-Fermi regime is the case where $\ep = 0$.   Thus, we split the Hamiltonian into
\be
H_\ep(\mf{u}) = \ep^2 \sum_{j=-1}^1|\na u_j|^2 + H_{TF}(\mf{u}).
\ee
A Thomas-Fermi solution is a measurable function $\mf{u}$ on $\Omega$ which solves the constrained variational problem:
\[
(P_{0})\quad \left\{ \begin{array}{l}
\inf \lt\{ \ml{E}_{TF}[\mf{u}] := \int_\Omega H_{TF}(\mf{u})\, dx \ \mid \ \mf{u}\in  (L^2(\Omega,\R_+))^3\rt\}\\
\mbox{ subject to the constraints } \
\ml{N}[\mf{u}]= N,\
\ml{M}[\mf{u}] = M.
\end{array}\right.
\]

It is expected that the Thomas-Fermi solutions are piecewise constant solutions consisting of one or two pure states in the following forms:
\bi
\I Nematic State (NS), if $\mf{u}=(0,u_0,0)$
\I Magnetic State (MS), if $\mf{u}= (u_1, 0,0)$ or $(0,0,u_{-1})$
\I Two-component State (2C), if $\mf{u} = (u_1,0,u_{-1})$
\I Three-component State (3C), if $\mf{u}=(u_1,u_0,u_{-1})$.
\ei
Here, $u_i$ above denotes a nonzero value.  We shall give a complete phase diagram of the ground states in this section and describe the ground state patterns in the next section. For a given total magnetization $M$, there exist two critical numbers $q_1,q_2$ such that we have the following description of the phase diagram:
\begin{itemize}
 \item For $c_s>0, q_2<q$, the Thomas-Fermi solution is a  $NS+MS$ mixed state.
 \item For $c_s>0, q_1<q<q_2$, the Thomas-Fermi solution is a  $NS+2C$ mixed state.
 \item For $c_s>0,q<q_1$, the Thomas-Fermi solution is a  $2C$ pure state.
 \item For $c_s<0,q<0$, the Thomas-Fermi solution is a  $MS+MS$ mixed state.
 \item For $c_s<0,0<q$, the Thomas-Fermi solution is a  $3C$ pure state.
\end{itemize}
Precisely, the notation $NS+MS$ means that there is a measurable set $U\subset \Omega$ such that
$$\mf{u}(x)= \mf{a}\,\chi_U(x) + \mf{b}\,\chi_{\Omega\setminus U}(x)$$
where $\chi_U$ and $\chi_{\Omega\setminus U}$ are characteristic functions and the vectors $\mf{a}=(0,u_0,0)$ and $\mf{b}=(u_1,0,0)$ or $(0,0,u_{-1})$ are two constant states.

\paragraph{Normalization and Notations}
We may divide $H_{TF}$ by $c_n$, and set $\alpha := c_s/c_n$ and {\em rename $q/c_n$ still by $q$}.
We consider $\Omega$ to be a bounded set with smooth boundary.  With this normalization, the Thomas-Fermi Hamiltonian becomes
\begin{equation}
H_{TF}(\mf{u}) := \frac{1}{2}({u_1^2+u_0^2+u_{-1}^2})^2 +  \frac{\alpha}{2}\big[2u_0^2(u_1- \textrm{sgn}(\alpha)\,u_{-1})^2+(u_1^2-u_{-1}^2)^2\big] + {q}(u_1^2+u_{-1}^2). \label{eq:3.3}
\end{equation}
We  denote the ratio $|U|/|\Omega|$ by $r$, the mass per unit volume $N/|\Omega|$ by $n$, and magnetization per unit volume $M/|\Omega|$ by $m$, respectively.

It is observed that the role exchanging  between $u_1$ and $u_{-1}$ does not change the form of \eqref{eq:3.3} and the constraint of the total mass \eqref{constraintN}. However, it changes the sign of the total magnetization \eqref{constraintM}. Because of this symmetric property, we only need to consider the case $m\ge 0$ without loss of generality .

\subsection{Antiferromagnetic BEC $(\alpha>0)$: $q_2<q$ implies $NS+MS$ state}\label{case:NS-MS}

\begin{thm}
  Suppose $0 < \alpha\le1$, $m\ge 0$.  Let
  $$
  q_2=\left( 1- \frac{1}{(\alpha+1)^{1/2}} \right)\, \left( n+\left(  (\alpha+1)^{1/2}-1\right)m\right).
  $$
 Then for $q > q_2$,  the global minimizer of the constrained variational problem ($P_{0}$) in the finite domain $\Omega$  takes the form
  $$ \mathbf u = \mathbf a\, \chi_U + \mathbf b\, \chi_{\Omega\backslash U}$$
  where $U\subset\Omega$ is a measurable set of size
  \be |U|= \frac{(\alpha+1)^{1/2}m}{n+ ((\alpha+1)^{1/2}-1)m} |\Omega|, \label{eq:vol1}\ee
  $$ \mathbf a =(\sqrt{ \frac{A}{(\alpha+1)^{1/2}} },0,0),\qquad \mathbf b=(0, \sqrt{A},0)$$
  and
  $$ A=  \left(n+\left( (\alpha+1)^{1/2}-1\right)m\right).$$
\end{thm}

\begin{proof}
\bn
\I First, we rewrite $H_{TF}$ as  the sum of several perfect squares:
\begin{align*}
2H_{TF} &= (u_1^2+u_0^2+u_{-1}^2)^2 + \alpha(u_1^2-u_{-1}^2)^2 + 2 \alpha u_0^2 (u_1-u_{-1})^2 + 2q (u_1^2 + u_{-1}^2)\\
&= (1+\alpha)u_1^4 + u_0^4 + (1+\alpha)u_{-1}^4 + 2(1-\alpha)u_1 ^2 u_{-1}^2 \\
& \quad+ 2 u_0^2\left[ (u_1^2+u_{-1}^2) + \alpha (u_1-u_{-1})^2\right] + 2q (u_1^2+u_{-1}^2)\\
&= \left( (1+\alpha)^{1/2} u_1^2 +u_0^2 + \frac{1-\alpha}{(1+\alpha)^{1/2}}u_{-1}^2 \right)^2 + \frac{4\alpha}{1+\alpha} u_{-1}^4 + 2q (u_1^2+u_{-1}^2) \\
& \quad + 2u_0^2\left[\left(1+\alpha -(1+\alpha)^{1/2}\right)u_1^2 + \left(1+\alpha -\frac{1-\alpha}{(1+\alpha)^{1/2}}\right) u_{-1}^2 - 2\alpha u_1 u_{-1}\right]
\end{align*}
In the last term, the quadratic form is non-negative because
\begin{align*}
&\left[\left(1+\alpha -(1+\alpha)^{1/2}\right)u_1^2 + \left(1+\alpha -\frac{1-\alpha}{(1+\alpha)^{1/2}}\right) u_{-1}^2 - 2\alpha u_1 u_{-1}\right]  \\
&=\left[ \left(  \left( (\alpha+1) - (\alpha+1)^{1/2} \right)^{1/2} u_1  -\frac{\alpha}{ \left( (\alpha+1) - (\alpha+1)^{1/2} \right)^{1/2} } u_{-1} \right)^2 \right.  \\
 &\left. \qquad \qquad+  \frac{  2\alpha +4 -2(\alpha+1)^{1/2} }{(\alpha+1) - (\alpha+1)^{1/2} }\, u_{-1}^2\right] \ge 0,
\end{align*}
 Here, we have used $\alpha+1-(\alpha+1)^{1/2}>0$ and $2\alpha +4 -2(\alpha+1)^{1/2} >0$ for $0<\alpha\le 1$.
\I
 Because of the constraints of total mass and total magnetization, we can convert the variational problem to an equivalent one by adding the terms involving $\beta_1 \int_\Omega (u_1^2 + u_0^2 + u_{-1}^2)\,dx$, $\beta_2 \int_\Omega (u_1^2-u_{-1}^2)\,dx $ and some constant $\int_\Omega A^2 \,dx$.   Namely, we can replace $H_{TF}$ with another energy density defined by
\be
W(\mf{u}) := H_{TF}(\mf{u}) + \frac{\beta_1}{2} (u_1^2+u_0^2+u_{-1}^2)+\frac{\beta_2}{2} (u_1^2-u_{-1}^2) + \frac{A^2}{2}, \label{W:NS-MS}
\ee
so that the minimization problem ($P_{0}$) is equivalent to
\[
\inf \int_\Omega W(\mf{u}) \, dx \mbox{ subject to } (\ref{constraintN}),(\ref{constraintM}).
\]
We shall choose the parameters $\beta_1,\beta_2$ and $A$ so that
\be
W(\mf{a})=W(\mf{b}) = 0, \ W(\mf{u})> 0 \mbox{ otherwise}.
\ee
We organize $W$ as
 \begin{align*}
 2W &=   \left( (\alpha+1)^{1/2} u_1^2 + u_0^2 + \frac{1-\alpha}{(1{{+}}\alpha)^{1/2}} u_{-1}^2 - A\right)^2 \nonumber \\
& + 2u_0^2 \left(  \left( (\alpha+1) - (\alpha+1)^{1/2} \right)^{1/2} u_1  -\frac{\alpha}{ \left( (\alpha+1) - (\alpha+1)^{1/2} \right)^{1/2} } u_{-1} \right)^2  \nonumber \\
& +  \frac{  2\alpha +4 -2(\alpha+1)^{1/2} }{(\alpha+1) - (\alpha+1)^{1/2} }\, u_{-1}^2 u_0^2   + \frac{4\alpha}{\alpha+1}\, u_{-1}^4 \label{eq:3.6}\\
& + \left( \beta_1+\beta_2+ 2q+ 2 A (\alpha+1)^{1/2}  \right)  u_1^2
      + (\beta_1 + 2 A) u_0^2
      + \left( \beta_1-\beta_2+{2q} + 2 A \frac{1-\alpha}{ (\alpha+1)^{1/2}}  \right)    u_{-1}^2.\nonumber
\end{align*}
\I
Now, we choose $\beta_1$ and $\beta_2$ to satisfy
$$
\left\{
\begin{array}{lcr}
 \beta_1+\beta_2 +{2q} + 2 A (\alpha+1)^{1/2} & =& 0 \\
 \beta_1+2 A &=& 0
\end{array}
\right. ,
$$
and the coefficient of $u_{-1}^2$ becomes
$$
  \beta_1 - \beta_2 +{2q} + 2 A \frac{1-\alpha}{(\alpha+1)^{1/2}}
  = 4A\left( \frac{1}{(\alpha+1)^{1/2}} -1\right) + 4q.
$$
We define $q_2$ by
\[
q_2 := A\left( 1- \frac{1}{(\alpha+1)^{1/2}} \right) .
\]
Now,
 \begin{align}
2W
&  = \left( (\alpha+1)^{1/2} u_1^2 + u_0^2 + \frac{1-\alpha}{(1+\alpha)^{1/2}} u_{-1}^2 - A\right)^2 \nonumber \\
& \qquad + 2u_0^2 \left(  \left( (\alpha+1) - (\alpha+1)^{1/2} \right)^{1/2} u_1  -\frac{\alpha}{ \left( (\alpha+1) - (\alpha+1)^{1/2} \right)^{1/2} } u_{-1} \right)^2  \nonumber \\
& \qquad +  \frac{  2\alpha +4 -2(\alpha+1)^{1/2} }{(\alpha+1) - (\alpha+1)^{1/2} }\,  u_0^2 u_{-1}^2   + \frac{4\alpha}{\alpha+1}\, u_{-1}^4
      + 4\left(q-q_2 \right)   u_{-1}^2. \label{eq:W-NS+MS}
\end{align}
When $q> q_2$, the coefficient of $u_{-1}$ is positive. Every term in the function $W$ is non-negative and the only zero of the function $W(u_1,u_0,u_{-1})$ in $\mathbb R^3_+$ is
$$ \mathbf a= \left(\sqrt{ \frac{A}{(\alpha+1)^{1/2}} },0,0\right) \qquad\textrm{and}\qquad \mathbf b=(0, \sqrt{A},0).$$

\I If $\mf{u}(x) = \mf{a}$ or $\mf{b}$ for $x$ a.e. in $\Omega$, then $\int_\Omega W(\mf{u}(x))\, dx = 0$.
On the other hand, since $W\ge0$, we get that if $\int_\Omega W(\mf{u}(x)) \, dx = 0$, then $W(\mf{u}(x)) = 0$ for $x$ a.e. in $\Omega$ and thus $\mf{u}(x) = \mf{a}$ or $\mf{b}$ a.e.~in $\Omega$. Therefore,
a minimizer of the  variational problem $\inf \int_\Omega W(\mathbf u)\,dx $  can be expressed as
\be \label{utwo}
 \mathbf u = \mathbf a\, \chi_U + \mathbf b\, \chi_{\Omega\backslash U}
 \ee
for some measurable set $U \subset \Omega$.
\I
Finally, we determine the constant $A$ and the measure $|U|$  by plugging (\ref{utwo}) into the two constraints (\ref{constraintN}), (\ref{constraintM}):
\begin{align*}
\frac{A}{(\alpha+1)^{1/2}} |U| + A (|\Omega|-|U|) &=N \\
\frac{A}{(\alpha+1)^{1/2}} |U| &= M.
\end{align*}
These lead to
\begin{align*}
A &= \frac{1}{|\Omega|} \left(N+(\alpha +1)^{1/2}M-M\right) \\
&=  \left(n+((\alpha +1)^{1/2}-1)m\right)\\
|U|&= \frac{(\alpha+1)^{1/2}m}{n+ ((\alpha+1)^{1/2}-1)m} |\Omega|.
\end{align*}
\en
\end{proof}

\subsection{Antiferromagnetic BEC $(\alpha>0)$: $q<q_1$ implies $2C$ state}\label{case:2C}
\begin{thm} \label{thm4.2}
Suppose $0<\alpha\le 1$ and $m\ge 0$. Let
\be \label{eq:q1}
q_1 =  \left(-n + \sqrt{n^2+\alpha m^2}\right).
\ee
Then for $q< q_1$ the global minimizer of the constrained variational problem ($P_{0}$)
 is the constant state
\be \label{u2C}
 \mathbf u =  \left(\sqrt{\frac{n+m}{2}},0,\sqrt{\frac{n-m}{2}}\right).
 \ee
\end{thm}
\begin{proof}
\bn
\I
Under the constraints of the total mass and total magnetization, the variational problem
$\inf \int_\Omega H_{TF}(\mf{u}(x))\,dx$ is equivalent to the variational problem $\inf \int_\Omega W(\mf{u}(x))\, dx$ under the same constraints, where
\begin{align}
W(u_1,u_0,u_{-1}) &:= H_{TF}(u_1,u_0,u_{-1}) + \frac{A^2}{2} + \alpha \frac{B^2}{2}  \nonumber\\
& \qquad - (A+{q})(u_1^2+u_0^2+u_{-1}^2) -\alpha B (u_1^2-u_{-1}^2). \label{W:2C}
\end{align}
The  goal is to show that  $W(\mf{u}) \ge 0$ for $0\le q < q_1$, and the only zero of $W$ satisfies (\ref{u2C}).  The constants $A$ and $B$ will be determined by the two constraints.
\I The strategy is to introduce two parameters $k_1$ and $k_2$ to make the coefficient of $u_0^4$ to be zero, to maximize the coefficient of $u_0^2$, and to make the rest to be non-negative:
\begin{align*}
2W &=
(u_1^2 + u_0^2+u_{-1}^2 - A )^2 + \alpha (u_1^2 - u_{-1}^2-B)^2 + 2\alpha u_0^2(u_1-u_{-1} )^2 - {2q}u_0^2\\
&=  (u_1^2 + k_1u_0^2  +u_{-1}^2 - A )^2  + \alpha (u_1^2 +k_2u_0^2 - u_{-1}^2-B)^2 \\
& \quad + (1-k_1^2-\alpha k_2^2) u_0^4 + 2 \left( (k_1-1)A+\alpha k_2 B - {q} \right)u_0^2\\
& \quad + 2 u_0^2 \left( \alpha (u_1-u_{-1})^2 + (1-k_1)(u_1^2+u_{-1}^2) - \alpha k_2 ( u_1^2 - u_{-1}^2)\right).
\end{align*}
By requiring
 \[
 \max_{k_1,k_2} \left( (k_1-1)A+\alpha k_2 B \right)  \mbox{ subject to } 1-k_1^2-\alpha k_2^2=0,
 \]
 we get
$$ k_1=\frac{A}{\sqrt{A^2+\alpha B^2}},\qquad k_2 = \frac{B}{\sqrt{A^2 + \alpha B^2}},$$
and
\[
 (k_1-1)A+\alpha k_2 B  =  -A + \sqrt{A^2+\alpha B^2}.
 \]
Now $2W$ becomes
\begin{align}
2W
&=  (u_1^2 + k_1u_0^2  +u_{-1}^2 - A )^2  + \alpha (u_1^2 +k_2u_0^2 - u_{-1}^2-B)^2 \nonumber \\
& \qquad + 2\left(-A + \sqrt{A^2+\alpha B^2} -{q}\right)u_0^2  \nonumber \\
& \qquad + 2 u_0^2 \left( \alpha (u_1-u_{-1})^2 + (1-k_1)(u_1^2+u_{-1}^2) - \alpha k_2 ( u_1^2 - u_{-1}^2)\right).\label{4.2-2W}
\end{align}
\I
For the last term, we claim that
$$ {{\alpha}}(u_1-u_{-1})^2 + (1-k_1)(u_1^2+u_{-1}^2) - \alpha k_2 ( u_1^2 - u_{-1}^2) \ge 0 \mbox{ for all }u_1,u_{-1}\in\mathbb R.
$$
This quadratic form of $u_1,u_{-1}$ can be re-expressed as
$$ (\alpha+1 - k_1 - \alpha k_2) u_1^2 + (\alpha+1-k_1+\alpha k_2)u_{-1}^2 -2\alpha u_1 u_{-1}.$$
By using $0 < \alpha \le 1$, we find that
\[
 \alpha+1 - k_1 - \alpha k_2 =
 \left(1-\frac{A}{\sqrt{A^2 + \alpha B^2}} \right) + \alpha \left( 1-\frac{B}{\sqrt{A^2+\alpha B^2}}\right) > 0,
\]
$$ \alpha+1 - k_1 + \alpha k_2>0,$$
and the discriminant of the quadratic form
\begin{eqnarray*}
&& (\alpha+1 - k_1 - \alpha k_2) (\alpha+1 - k_1 + \alpha k_2) -\alpha^2 \\
&=& \frac{(1+\alpha) \alpha^2 B^4 }{ (A^2+\alpha B^2) \left[ 2 A^2 + \alpha B^2 + 2A\sqrt{A^2+\alpha B^2} \right]} > 0
\end{eqnarray*}
This proves the claim.
\I For the $u_0^2$ term,
we define the constant $q_1$ by
$$ q_1:=  \left( -A + \sqrt{A^2+\alpha B^2} \right),$$
then the $u_0^2$ term becomes
\[
{2}\left(q_1-q\right) u_0^2.
\]
When $q < q_1$, this term is positive.
\I We have seen that for $q < q_1$, all the terms of $W$ are nonnegative and $W=0$ only when $\mf{u}=\left(\sqrt{\frac{A+B}{2}},0,\sqrt{\frac{A-B}{2}}\right)$.
\I By plugging this constant state into the two constraints (\ref{constraintN}), (\ref{constraintM}), we get
\[
A = n, \ \ B = m.
\]
\en
\end{proof}

\subsection{Antiferromagnetic BEC $(\alpha>0)$:  $q_1<q<q_2$ implies $NS+2C$ state}\label{case:NS-2C}

\begin{thm}
 Suppose  $0<\alpha\le 1$ and $q_1<q<q_2$. \\
 If $m > 0$, then the global minimizer of the variational problem $(P_{0})$ takes the form
 \begin{equation}\label{eq:3.11}
  \mathbf u = \mathbf a \,\chi_U + \mathbf b\,\chi_{\Omega\backslash U}
  \end{equation}
 where
 \begin{equation}\label{eq:3.12}
  \mathbf a =(\sqrt{\frac{A+B}{2}},0,\sqrt{\frac{A-B}{2}}) \quad\textrm{and}\quad \mathbf b=  (0,\sqrt[4]{A^2+\alpha B^2},0),
  \end{equation}
 $$ A=n + (r-1)\tq \quad\textrm{and}\quad B=\frac{m}{r},$$
where  $r:= |U|/|\Omega|$ satisfies
 \be 2q^2 r^3 + \left( 2qn - q^2 \right) r^2 -\alpha m^2 =0. \label{eq:vol2}\ee
 If $m=0$, then the global minimizer of the variational problem $(P_{0})$ has the form
 \begin{equation}\label{eq:3.13}
  \mathbf u = \mathbf b\,\chi_{\Omega}\qquad \textrm{where }\quad\mathbf b=  (0,\sqrt{n},0).
 \end{equation}
\end{thm}
\begin{proof}
\bn
\I Following the first two steps in Section \ref{case:2C}, we now choose $A$ and $B$ to cancel the $u_0^2$ term in (\ref{4.2-2W}), i.e. we require
\be \label{4.3-u_0^2}
-A+\sqrt{A^2+\alpha B^2}={q}.
\ee
Then
\begin{align}
2W(u_1,u_0,u_{-1}) &=
(u_1^2 +k_1 u_0^2+u_{-1}^2-A)^2 + \alpha (u_1^2 + k_2 u_0^2 - u_{-1}^2-B)^2   \label{W:NS-2C}\\
&\quad + 2 u_0^2 \left[ \alpha (u_1-u_{-1})^2  +(1-k_1) (u_1^2+u_{-1}^2)  -\alpha k_2 (u_1^2 - u_{-1}^2)\right].\nonumber
\end{align}
We have seen in Step 3 of Section \ref{case:2C} that the quadratic part in the last term is non-negative and it is zero only when $u_1=u_{-1}=0$.  Thus, from (\ref{W:NS-2C}), $W \ge 0$ and there are two roots for $W=0$:
\begin{enumerate}
 \item   $u_0=0$ and from (\ref{W:NS-2C}) $u_1,u_{-1}$ satisfy the system
 \be \label{eq:2C-AB}
         \left\{
         \begin{array}{rcl}
          u_1^2 + u_{-1}^2 &=& A, \\
          u_1^2 - u_{-1}^2 &=& B.
         \end{array}
         \right.
\ee
         This is a  2C state.  It takes the form
         $$ \mathbf u = \left(\sqrt{ \frac{1}{2}(A+B)},\,0,\sqrt{\frac{1}{2}(A-B)}\right)\equiv \mathbf a.$$

 \item $ u_1 = u_{-1} = 0$, but $u_0 \ne 0$.  From (\ref{W:NS-2C}),
         $$
          u_0^2 = \sqrt{A^2+\alpha B^2}.
         $$
         This is a NS state.  It takes the form
         $$  \mathbf u = (0,\sqrt[4]{A^2+\alpha B^2},0)\equiv \mathbf b.$$
\end{enumerate}
In this case, the minimizers of  the variational problem
$$ \inf \int_\Omega W(u_1,u_0,u_{-1})\,dx$$
 take the form
$$ \mathbf u = \mathbf a \chi_U + \mathbf b \chi_{\Omega\backslash U},$$
where $U$ is any measurable set in $\Omega$ with relative size $r:= |U|/ |\Omega|$.
\I Our remaining task is to show the existence of $A$, $B$ and $r$ for $q_1 < q < q_2$ from the condition (\ref{4.3-u_0^2}) and the two constraints (\ref{constraintN}) (\ref{constraintM}) and some natural inequality  constraints. We list them below.
\begin{align}
 -A+ \sqrt{A^2+\alpha B^2}  &=  q   \label{eq:4.3-1}\\
 r A + (1-r) \sqrt{A^2+\alpha B^2} &= n \label{eq:4.3-2}\\
 r B &= m \label{eq:4.3-3}\\
 0\le  r \le 1, \  0\le B \le A, \quad & q_1 \le q \le q_2 . \label{eq:4.3-4}
\end{align}
The inequality $0\le B \le A$ is due to ({eq:2C-AB}).

The first two equations give
$$ A=n + (r-1)\tq.$$
The third equation leads to
$$B=\frac{m}{r}.$$
Substituting these two into the first equation, we obtain
$$ 2q^2 r^3 + \left( 2\tq n -q^2 \right) r^2 -\alpha m^2 =0. $$
We may rewrite it as the following dimensionless form
\be \label{eq:xeq}
2\bq^2 r^3 + \left( 2\bq  -\bq^2 \right) r^2 -\eta =0,
\ee
where $\bq := q/n$, $\eta = \alpha m^2/n^2$, $0\le \eta \le 1$.

Let us  express the condition $B\le A$  in terms of $r$, $\bq$ variables:
\[
 \frac{m}{r}=B \le A= n+ {r q}- q \quad\Leftrightarrow\quad \bq r^2 + (1-\bq)r -\frac{m}{n} \ge 0
\quad \Leftrightarrow\quad
\bq \le h(r):= \frac{1-\frac{m}{n}\frac{1}{r}}{1-r},
\]
for $0\le r \le 1$.
So our goal is to solve (\ref{eq:xeq}) for $x \in [0,1]$ for given $\bq\in [\bq_1 , \bq_2]$ and satisfying $\bq \le h(r)$.   Here, $\bq_i := q_i/n$, $i=1,2$.
\I
We  rewrite (\ref{eq:xeq}) as a quadratic equation for $\bq$:
\be
(2r^3-r^2) \bq^2 + 2r^2 \bq -\eta=0.
\ee
There are two branches of solutions for $\bq$:
\begin{align}
Q_1(r) &= \frac{-r^2 + \sqrt{r^4+\eta (2r^3-r^2)}}{2r^3-r^2}, \label{eq:Q1}\\
Q_2(r) &= \frac{-r^2 - \sqrt{r^4+\eta (2r^3-r^2)}}{2r^3-r^2}.  \nonumber
\end{align}
Since the turning point $r_0$ satisfying the equation
$r^2+\eta (2r-1) = 0$,  we find $r_0 = -\eta + \sqrt{\eta^2+\eta}$.  Furthermore, we have $0\le r_0 \le \sqrt{2}-1 < 1/2$ because $0\le \eta \le 1$.

By direct calculation, we get that $Q_1(r)> 0$ and decreasing on  $r > r_0$, and  $Q_2(r) > 0$ and increasing on $r_0 < r < 1/2$.  Furthermore, $Q_2(r)\to \infty$ as $r\to 1/2-$.  We plot the solution curve with $n=1,m=0.2,\alpha=0.8$ and the corresponding $\eta = 0.0032$ in Figure \ref{fig:Q1Q2h} .
We also notice that $h'(r) > 0$ for $0< r < 1$ and $h(1-) = \infty$.  Here, we have used  $m/n \le 1$.

\begin{figure}[ht]
\begin{center}
\includegraphics[scale=0.8]{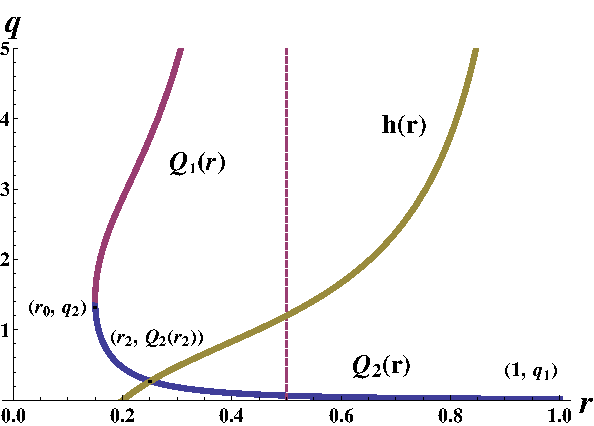}
\caption{$n=1,m=0.2,\alpha=0.8,\eta=0.0032$}
\label{fig:Q1Q2h}
\end{center}
\end{figure}

\I
We claim that there is no solution for  $Q_2(r) \le h(r)$ with $r_0 \le r < 1/2$.  We calculate
\begin{align*}
Q_2(r) - h(r) &= \frac{1 + \sqrt{1+ \left(\frac{2}{r}-\frac{1}{r^2}\right)\alpha\frac{m^2}{n^2}}}{1-2r}-\frac{r-\frac{m}{n}}{r-r^2}\\
&= \frac{r+\left(\frac{1}{r}-1\right)\frac{m}{n} + (1-r)\sqrt{1+ \left(\frac{2}{r}-\frac{1}{r^2}\right)\alpha\frac{m^2}{n^2}}}{(1-2r)(1-r)} > 0,
\end{align*}
for $r_0 \le r < 1/2$.  Thus, there is no admissible solution on the branch $\bq=Q_2(r)$.
\I Given $\bq \in [\bq_1,\bq_2]$, we look for  $r\in [r_0,1]$ such that $\bq=Q_1(r)$ and $\bq \le h(r)$.
We first notice that $\bq_1 = Q_1(1)$ from (\ref{eq:Q1}) and (\ref{eq:q1}).
On $(r_0,1)$, the branch $\bq = Q_1(r)$  is strictly decreasing, and the function $\bq=h(r)$ is strictly increasing.  Thus,   there exists  a unique $r_2\in (r_0, 1)$ such that $Q_1(r_2)=h(r_2)$, because  $Q_1(r_0) =  Q_2(r_0) > h(r_0)$ and   $Q_1(1)-h(1-) = -\infty$.
Indeed, at $r_2$,  we have $A=B$.   Then from (\ref{eq:4.3-1}), (\ref{eq:4.3-2}),(\ref{eq:4.3-3}), we get
\begin{align*}
Q_1(r_2) &= \frac{\sqrt{1+\alpha}-1}{\sqrt{1+\alpha}} \left(1+(\sqrt{1+\alpha}-1)\frac{m}{n}\right) = \bq_2 = h(r_2),\\
r_2 &= \frac{\sqrt{1+\alpha}}{\sqrt{1+\alpha} -1 +\frac{n}{m}}
\end{align*}
For $r \in (r_2,1)$, we have $Q_1(r) < h(r)$ because $Q_1(\cdot)$ is strictly decreasing and $h(\cdot)$ is strictly increasing.
Now, we know $Q_1(1)=\bq_1$ and $Q_1(r_2) = \bq_2$ and $Q_1(\cdot)$ is strictly decreasing
in $(r_2)$, we get for every $\bq \in (\bq_1,\bq_2)$, there exists a unique $r\in (r_2,1)$ such that $\bq=Q_1(r)$ and $Q_1(r) <h(r)$.
\I Lastly, we discuss the case of $m=0$. We proceed the Step 1 and find that the  global minimizer still has the form \eqref{eq:3.11} and \eqref{eq:3.12}. We obtain $x, A$ and $B$ by solving  \eqref{eq:4.3-1}, \eqref{eq:4.3-2} and \eqref{eq:4.3-3} and get
$$ r=0,\qquad A = n-q\qquad\textrm{and}\qquad B=\sqrt{\frac{(2n-q)q}{\alpha}}.$$
This gives \eqref{eq:3.13}.
\en
\end{proof}

\begin{remark}
For antiferromagnetic BEC $(\alpha>0)$, the 2C components $u_1$, $u_{-1}$ are suppressed by the value $q$. Consider the case that the total magnetization $m>0$. In general, the larger value $q$ is , the smaller values $u_1$ and $u_{-1}$ are. On the other hand,  as $q$ increases larger than $q_1$ the component $u_0$ begins to increase. When $q$ is larger than $q_2$, the component $u_{-1}$ vanishes and the ground state is $NS+MS$.
\end{remark}

\subsection{Ferromagnetic BEC $(\alpha<0)$: $q<0$ implies $MS+MS$ state}\label{case:MS-MS}
\begin{thm}
  Suppose  $-1<\alpha<0$, $m\ge 0$ and $q<0$.  Then the global minimizers of the constrained variational problem $(P_{0})$ take the form
 $$ \mathbf u = \mathbf a \,\chi_U + \mathbf b\,\chi_{\Omega\backslash U},$$
 where $U$ is a measurable set of size
 \be |U|=\frac{1}{2}\left( 1+ \frac{m}{n}\right)|\Omega| \label{eq:vol3}\ee
  and
 $$ \mathbf a = (\sqrt{n},0,0)\quad\textrm{and} \quad\mathbf b=(0,0,\sqrt{n}).$$
\end{thm}
\begin{proof}
\bn
\I First, we notice that
\begin{align*}
2H_{TF} &=  (u_1^2 + u_0^2+u_{-1}^2)^2 + \alpha (u_1^2 - u_{-1}^2)^2 + 2\alpha u_0^2(u_1 + u_{-1})^2+ {2q}(u_1^2+u_{-1}^2) \\
&= (1+\alpha)  (u_1^2 + u_0^2+u_{-1}^2)^2 -\alpha \left( u_0^2-2 |u_1| |u_{-1}|\right)^2 +{2q}(u_1^2+u_{-1}^2).
\end{align*}
Here we have used the following algebraic identity:
$$ 2 u_0^2 (|u_1| + |u_{-1}| )^2 + (u_1^2-u_{-1}^2)^2 - (u_1^2 + u_0^2+u_{-1}^2)^2 = -(u_0^2-2|u_1| |u_{-1}|)^2.$$
\I
Because of the constraint of total mass, we can convert our variational problem
to an equivalent one by adding $\int_\Omega[-2(A(1+\alpha)+q) (u_1^2+u_0^2+u_{-1}^2) + (1+\alpha) A^2]\, dx$ to the functional $2\int H_{TF}\, dx$.
That is, the new energy density
\begin{align}
2W &:= 2H_{TF} -\left(2A(1+\alpha)+2q \right)  (u_1^2+u_0^2+u_{-1}^2) + (1+\alpha)A^2 \nonumber  \\
&= (1+\alpha) (u_1^2 +u_0^2+ u_{-1}^2 -A)^2
  - \alpha \left( u_0^2-2 |u_1| |u_{-1}|\right)^2 -{2q}u_0^2 \label{W:MS-MS}
\end{align}

\I Since $\alpha<0$ and $q<0$,  every term of the function $W$ is non-negative. The function equals to zero if and only if
$$ \mathbf u=(\sqrt{A},0,0)\equiv\mathbf a \qquad\textrm{or}\qquad\mathbf u=(0,0,\sqrt{A})\equiv \mathbf b.$$
\I A measurable function $\mf{u}(x)$ on $\Omega$ satisfies $\int_\Omega W(\mf{u}(x))\, dx = 0$ if and only if there is a measurable set $U\subset \Omega$ such that
$$ \mathbf u = \mathbf a \,\chi_U + \mathbf b\,\chi_{\Omega\backslash U}.$$
\I By plugging such function $\mf{u}(x)$ into the two constraints (\ref{constraintN}) and (\ref{constraintM}), we  get
\[
A = n, \ |U| = \frac{1}{2}\left( 1+ \frac{m}{n}\right)|\Omega|.
\]
 \en
\end{proof}

\subsection{Ferromagnetic BEC $(\alpha<0)$: $0<q$ implies $3C$ state}\label{case:3C}

\begin{thm}
  Suppose $-1<\alpha<0$, $m\ge 0$  and $q>0$. Then the constrained variational problem $(P_{0})$ has a unique global minimizer
 $$ \mathbf u = \left( u_1,u_0,u_{-1} \right)$$
 where
 \begin{eqnarray*}
   u_1&=&\frac{q+b}{2q} \left[ n+ \frac{1}{\alpha}  \left( \frac{q}{2} - \frac{b^2}{2q} \right)\right]^{1/2 } \\
    u_0 &=& \left[\frac{q^2-b^2}{2q^2} n - \frac{q^2+b^2}{2q^2} \frac{1}{\alpha} \left( \frac{q}{2} - \frac{b^2}{2q} \right)\right]^{1/2}\\
   u_{-1} &=& \frac{q-b}{2q} \left[ n+ \frac{1}{\alpha} \left( \frac{q}{2} - \frac{b^2}{2q} \right) \right]^{1/2 }.
 \end{eqnarray*}
 The value $b$ is the unique root in $(\sqrt{q^2+2\alpha q n},q)$ of the cubic equation
 $$b^3 - (q^2+ 2\alpha q n) b + 2 \alpha q^2 m=0.$$
\end{thm}
\begin{proof}
\bn
\I
As in the proof of the previous section,
\begin{align*}
2H_{TF} &:= (u_1^2 +u_0^2+ u_{-1}^2)^2 + \alpha \left( 2 u_0^2 (|u_1| + |u_{-1}| )^2 + (u_1^2-u_{-1}^2)^2 \right) + 2q(u_1^2+u_{-1}^2) \\
&= (1+\alpha) (u_1^2 +u_0^2+ u_{-1}^2 )^2
    - \alpha \left( u_0^2-2 |u_1| |u_{-1}| \right)^2 + 2q (u_1^2+u_{-1}^2).
\end{align*}
Due to the two constraints, we can add $-2a(u_1^2+u_0^2+u_{-1}^2)$ and $-2b(u_1^2-u_{-1}^2)$ for some constants $a,b$ to the above expression without changing the constrained variational problem. We obtain
\begin{align*}
2H_{TF}&-2a(u_1^2+u_0^2+u_{-1}^2)-2b(u_1^2-u_{-1}^2)\\
&= (1+\alpha) (u_1^2 +u_0^2+ u_{-1}^2 )^2
    - \alpha \left( u_0^2-2 |u_1| |u_{-1}| \right)^2 \\
 &\quad  + 2(q-a-b) u_1^2 - 2a u_0^2 + 2(q-a+b) u_{-1}^2 \\
&=  (1+\alpha) (u_1^2 +u_0^2+ u_{-1}^2 )^2
    - \alpha \left( u_0^2-2 |u_1| |u_{-1}| \right)^2 \\
& \quad+ 2\left[ (q-a-b)^{1/2}|u_1| - (q-a+b)^{1/2} |u_{-1}|  \right]^2\\
& \quad +4  (q-a-b)^{1/2}(q-a+b)^{1/2}  |u_1| |u_{-1}|    -2 a|u_0|^2
\end{align*}
There will be two relations to determine $a$ and $b$.    First, we introduce  the relation
\be \label{eq:ab}
 a= (q-a-b)^{1/2}(q-a+b)^{1/2} .
\ee
This relation is equivalent to
\be \label{eq:a1}
 a = \frac{q}{2} - \frac{b^2}{2q},
 \ee
and leads to
$$ q-a-b=\frac{q}{2} +\frac{b^2}{2q} -b = \frac{(q-b)^2}{2q}\ge 0,$$
$$  q-a+b=\frac{q}{2} +\frac{b^2}{2q} +b = \frac{(q+b)^2}{2q}\ge 0.$$
Since $a$ should be  non-negative from (\ref{eq:ab}), we have
\be \label{eq:b1}
a = \frac{q}{2} - \frac{b^2}{2q}\ge 0
\quad\Rightarrow\quad
|b|\le q.
\ee
With this choice of $a,b$ satisfying the relations (\ref{eq:a1}) and (\ref{eq:b1}),  we define
\begin{align}
2W &= 2H_{TF}-2[a+(1+\alpha)n](u_1^2+u_0^2+u_{-1}^2)-2b(u_1^2-u_{-1}^2)  +(1+\alpha)n^2 -\frac{a^2}{\alpha} \nonumber\\
&= (1+\alpha) (u_1^2 +u_0^2+ u_{-1}^2 -n)^2
    - \alpha \left( u_0^2-2 |u_1| |u_{-1}|  + \frac{a}{\alpha}\right)^2  \nonumber\\
 & \quad  + \frac{1}{q}\left[ (q-b) |u_1| -(q+b) |u_{-1}|  \right]^2.  \label{W:3C}
\end{align}
and  the original constrained variational problem is equivalent to  $\inf \int_\Omega  W(\mf{u}(x))\, dx$.
\I
For any given $q>0$,  $W$ of (\ref{W:3C}) has a unique minimizer   $(u_1,u_0,u_{-1})$.  This leads to  the following algebraic system for $(u_1,u_0,u_{-1})$:
\begin{align}
 u_1^2 + u_0^2 + u_{-1}^2 &= n \label{eq:4.5-1} \\
 u_0^2-2|u_1| |u_{-1}| &= -\frac{a}{\alpha} \label{eq:4.5-2} \\
  (q-b) |u_1| - (q+b) |u_{-1}|  &=0 \label{eq:4.5-3}
 \end{align}
For any fixed $b$, we solve this algebraic system for $(u_1,u_0,u_{-1})$:
\begin{align}
|u_1| &=  \frac{q+b}{2q} \left( n+ \frac{a}{\alpha} \right)^{1/2 }   \label{eq:4.6-1}\\
|u_{-1}| &= \frac{q-b}{2q} \left( n+ \frac{a}{\alpha} \right)^{1/2 }\label{eq:4.6-2} \\
 |u_0| &= \left[\frac{q^2-b^2}{2q^2} n + \frac{q^2+b^2}{2q^2} \frac{a}{\alpha}\right]^{1/2}. \nonumber
\end{align}
\I  Our remaining task is find a relation to determine $b$.  Since the constant state $(u_1,u_0,u_{-1})$ is the unique minimizer of $W$, we apply the constraint of total magnetization to this constant state and find
\be
 u_1^2- u_{-1}^2 =m. \label{eq:4.5-4}
\ee
 From (\ref{eq:4.5-4}),(\ref{eq:4.6-1}) and (\ref{eq:4.6-2}), we obtain
$$ \left( n + \frac{a}{\alpha}\right) \left[ \left( \frac{q+b}{2q}\right)^2 - \left( \frac{q-b}{2q}\right)^2 \right] = m.$$
Plugging (\ref{eq:a1}) into this equation, we obtain
$$ b\,\left( n+ \frac{1}{2\alpha q} (q^2-b^2) \right) = mq,$$
or equivalently
\[
g(b) = 0,
\]
where
$$ g(b):=b^3 - (q^2+ 2\alpha q n) b + 2 \alpha q^2 m=0.$$
This is the equation to determine $b$.  In addition, there are other natural constraints that $b$ should satisfy.
In fact,
subtracting  (\ref{eq:4.5-2}) from (\ref{eq:4.5-1}), we obtain
\be \label{cond:a1}
 n+ \frac{a}{\alpha}=(|u_1|+|u_{-1}|)^2 \ge 0.
 \ee
Combing this with (\ref{eq:a1}) yields
\be \label{cond:b1}
 b^2 \ge q^2 + 2\alpha q n.
\ee
On the other hand, we have $|b|\le q$ from (\ref{eq:b1}).  Therefore, $b$ must lie in the interval $[\sqrt{q^2+2\alpha q n},q]$.
\I We claim that $g(b)=0$ has a unique root in $[\sqrt{q^2+4\alpha q n},q]$.
Because we have
$$ g(\sqrt{q^2+2\alpha q n})= \left({q^2+2\alpha q n}\right)^{3/2} - (q^2+ 2\alpha q n) \sqrt{q^2+2\alpha q n} + 2 \alpha q^2 m = 2\alpha q^2 m<0,$$
$$ g(q) = q^3 - (q^2+ 2\alpha q n) q + 2 \alpha q^2 m=2\alpha q (m-n) >0$$
and
$$
 g'(b) = 3 b^2 - (q^2+ 2\alpha n q)
  = 2 b^2 +b^2 - (q^2+2\alpha n q)
  > 0,
$$
The function $g$ is strictly monotone on the interval $[\sqrt{q^2+2\alpha q n},q]$. Therefore, it has
a unique root between on $[\sqrt{q^2+2\alpha q n},q]$.
\I
From the above discussion, we conclude there exists a unique $3C$ state when $c_s<0$ and $q>0$.
\en
\end{proof}

\begin{remark}
 Suppose $\alpha<0$ and $q=0$. In this case, following step 1 of Section \ref{case:3C} with $a=b=0$, we set
 $$\aligned
    2W:=& 2H_{TF} - 2(1+\alpha)n(u_1^2+u_0^2+u_{-1}^2) + (1+\alpha)^2n^2 \\
      =& (1+\alpha) (u_1^2 +u_0^2+ u_{-1}^2 -n )^2
    - \alpha \left( u_0^2-2 |u_1| |u_{-1}| \right)^2 .
  \endaligned
 $$
  The constrained variational problem ($P_{0}$) is equivalent to
the  variational problem
  $$ \inf \int_\Omega W(u_1,u_0,u_{-1})\,dx.$$
Its constrained minimizer $\mf{u}=(u_1,u_0,u_{-1})$ satisfies
 $$
 \left\{\begin{array}{rcl}
  u_1^2(x) + u_0^2(x) + u_{-1}^2(x) &=& n \\
  u_0^2(x) - 2|u_1(x)||u_{-1}(x)| &=& 0
 \end{array}\right.
 $$
 for almost all $x\in \Omega$.  This gives $|u_1(x)|+|u_{-1}(x)|=\sqrt{n}$.
 Let us call $u_1^2(x)-u_{-1}^2(x) = \tilde{m}(x)$.  From the conservation of total magnetization, we should require
 $\int \tilde{m}(x)\, dx = M$.
We have
\[
(|u_1(x)|-|u_{-1}(x)|)(|u_1(x) +|u_{-1}(x)|) = \tilde{m}(x).
\]
Thus, we get
\[
u_1(x)=\frac{n+\tilde{m}(x)}{2\sqrt{n}},  u_{-1}(x)=\frac{n-\tilde{m}(x)}{2\sqrt{n}}.
\]
Since $\tilde{m}(x)$ can be any arbitrary bounded measurable function with $|\tilde{m}(x)|\le n$, there are infinite many Thomas-Fermi solutions in this case.
\end{remark}

\subsection{BEC with $(\alpha=0)$}

\begin{thm}
 Suppose $\alpha=0$. Then the global minimizer of the constrained variational problem ($P_{0}$) in a finite domain $\Omega$  is in either one of the following cases:
 \begin{itemize}
  \item[(i)] If $q=0$, then a minimizer takes the form $\mf{u}(x)=(u_1(x),u_0(x),u_{-1}(x))$  such that
  \begin{equation}\label{eq:3.35}
  u_1^2(x)+u_0^2(x)+u_{-1}^2(x)=n\quad\textrm{for almost all }x\in\Omega \textrm{ with }\int_\Omega u_1^2 - u_{-1}^2\,dx = M.
  \end{equation}
  \item[(ii)] If $q>0$, then a minimizer takes the form $\mf{u}(x)=(u_1(x),u_0(x),0)$  such that
  \begin{equation}\label{eq:3.36}
  u_1^2(x)+u_0^2(x)=n\quad\textrm{for almost all }x\in\Omega \textrm{ with }\int_\Omega u_1^2\,dx = M.
  \end{equation}
  \item[(iii)] If $q<0$, then a minimizer takes the form $\mf{u}(x)=(u_1(x),0,u_{-1}(x))$  such that
  \begin{equation}\label{eq:3.37}
   u_1^2(x)+u_{-1}^2(x)=n\quad\textrm{for almost all }x\in\Omega \textrm{ with }\int_\Omega u_1^2-u_{-1}^2\,dx = M.
   \end{equation}
 \end{itemize}
\end{thm}

\begin{proof}
\begin{enumerate}
 \item When $\alpha = 0$ and $q=0$, we have $H_{TF}=\frac{1}{2}(u_1^2+u_{0}^2+u_{-1}^2)^2$.  We set
  \begin{align*}
  W(\mathbf u) & = H_{TF}(\mathbf u) -n(u_1^2+u_0^2+u_{-1}^2) +\frac{n^2}{2}
      = \frac{1}{2}(u_1^2+u_0^2+u_{-1}^2-n)^2.
  \end{align*}
Then  the constrained variational problem $\left( \inf \int_\Omega W(u_1,u_0,u_{-1})\,dx \right)$ is equivalent to the original one and   its minimum is characterized by \eqref{eq:3.35}.
 \item When $\alpha = 0$ and $q > 0$, we follow the proof of Section \ref{case:NS-MS}.
In this case, $q_2 = 0$ and $W$ of (\ref{eq:W-NS+MS}) becomes
\[
2W =    \left( u_1^2 + u_0^2 +  u_{-1}^2 - A\right)^2   + 4q   u_{-1}^2.
\]
Because $W\ge 0$ and the constrained minimization can only occur when $\int W(\mf{u}(x))\, dx = 0$.  This implies $W(\mf u (x))=0$ for almost all $x\in \Omega$.  This forces $u_{-1}(x) = 0$ and $u_1^2(x)+u_0^2(x)+u_{-1}^2(x)=A$.  From the total mass constraint, we have to choose $A=n$.  %
\item When $\alpha = 0$ and $q < 0$, we follow the proof of Section \ref{case:MS-MS}.
In this case, the normalized energy density $W$ of (\ref{W:MS-MS}) becomes
 \begin{align*}
2W &=  (u_1^2 +u_0^2+ u_{-1}^2 -A)^2 -{2q}u_0^2.
\end{align*}
Similar to the argument of the previous step, we find $u_0(x)\equiv 0$ and $A=n$.
\end{enumerate}
\end{proof}

\section{$\Gamma$-convergence}\label{sec:gamma_conv}
\subsection{Interfacial and boundary energy functional}
The Thomas-Fermi solutions found in the last section are  not unique in general.  In fact, the pure states (Sections \ref{case:2C}, \ref{case:3C}) are unique, while the mixed states (Sections \ref{case:NS-MS}, \ref{case:NS-2C}, \ref{case:MS-MS}) are not unique. In the case of mixed state, which has the form: $\mf{u}(x) = \mf{a}\chi_{U}(x)+\mf{b}\chi_{\Omega\setminus U}(x)$,   only the ratio $|U|/|\Omega|$ is determined,  but the measurable set $U$ can be arbitrary.   It has been pointed out by Gurtin that interfaces are allowed to form without changing the bulk energy $\int_{\Omega} H_{TF}(\mf{u}(x))\, dx$~\cite{G1,G2}.  To select a physical solution, we adopt the $\Gamma$-convergence theory, which introduces an interfacial energy functional to penalize the formation of interfaces.  This interfacial energy functional  is the $\Gamma$-limit of the next-order expansion of the energy functional $\ml{E}_\ep[\mf{u}]$ as $\ep \to 0$.
To be precise, let us recall that 
\[
\ml{E}_\ep [\mf{u}] := \int_\Omega \ep^2 |\grad \mf{u}|^2 + H_{TF}(\mf{u})\, dx.
\]
We write the domain of $\ml{E}_\ep$ to be
\[
\ms{A} := \{ \mf{u}\in (H^1_0(\Omega,\R_+))^3 \cap (L^4(\Omega,\R_+))^3 |\, \ml{N}[\mf{u}]=N, \ml{M}[\mf{u}]=M\}.
\]
We expect
\[
\lim_{\ep \to 0} \inf \ml{E}_\ep[\mf{u}] = \inf \ml{E}_0[\mf{u}],
\]
and thus  look for  the minimizers  of $\ml{E}_0$ (i.e. the Thomas-Fermi solutions).  Let us call them
\[
\ms{A}_0 = \mbox{ arg min} \{ \ml{E}_0[\mf{u}] | \,\mf{u}\in (L^2(\Omega,\R_+))^3, \ml{N}[\mf{u}]=N, \ml{M}[\mf{u}]=M\},
\]
and the corresponding minimal  energy $E_0$.  For mixed states, the set $\ms{A}_0$, which is not a singleton, can also be expressed as
\be \label{eq:A0}
\ms{A}_0 = \{ \mf{u}=\mf{a}\chi_U + \mf{b}\chi_{\Omega\setminus U} |\,  r= |U|/|\Omega|, \mf{a},\mf{b} \mbox{ are given in } \eqref{eq:vol1}, \eqref{eq:vol2} \mbox{ or } \eqref{eq:vol3}\}.
\ee  
We then define the next order energy functional $\ml{G}_\ep:\ms{A}\to \R$ to be
\[
\ml{G}_\ep[\mf{u}] = \frac{\ml{E}_\ep [\mf{u}] - E_0}{\ep}.
\]
From the previous section, this functional  has the form
\be
\ml{G}_\ep [\mf{u}] := \int_\Omega \ep |\grad \mf{u}|^2 + \frac{1}{\ep}W(\mf{u})\, dx,
\ee
where $W$ is given in (\ref{W:NS-MS}), (\ref{W:NS-2C}), (\ref{W:MS-MS})  which has the properties: $W(\mf{u})\ge 0$ and $W(\mf{u})=0$ if and only if $\mf{u}=\mf{a}$ or $\mf{b}$.
We expect that 
$$\lim_{\ep \to 0}\inf_{\mf{u}\in \ms{A}}  \ml{G}_\ep[\mf{u}] = \inf_{\mf{u}\in \ms{A}_0} \ml{G}_0[\mf{u}]
.$$
Here, the functional $\ml{G}_0$ is so-called the $\Gamma$-limit of $\ml{G}_\ep$, where the precise definition will be given in Theorem \ref{Thm:Gamma-limit}.  We will prove that  $\ml{G}_0: \ms{A}_0 \to \R$ is given by: 
$$
\ml G_0[\mathbf u] =\left\{\begin{array}{ll}
  2  g(\mathbf a,\mathbf b)\, \textrm{Per}_\Omega(\mathbf u=\mathbf a)& \textrm{for } \mathbf u=\mathbf a \chi_U + \mathbf b\chi_{\Omega\backslash U}\in \ms{A}_0\cap \left(BV(\Omega;\mathbb R_+)\right)^3,\\
   \,  + 2g(0,\mathbf a)\,\mathcal H^2(\{x\in\partial \Omega: \mathbf u(x)=\mathbf a\})   
   & \\
    \, +2g(0,\mathbf b)\,\mathcal H^2(\{x\in\partial \Omega: \mathbf u(x)=\mathbf b\}) &
    \\
      +\infty &  \textrm{otherwise} .
\end{array}\right.
$$
where
$$
g(\mathbf v,\mathbf u)= \inf\left\{\ \int_0^1 \sqrt{W(\boldsymbol \gamma(t))}\,|\boldsymbol \gamma'(t)|\,dt:\,
       \boldsymbol\gamma:[0,1]\to\mathbb R^3_+ \textrm{ Lipchitz continuous}, \,\boldsymbol \gamma(0)=\mathbf v,\,\boldsymbol\gamma(1)=\mathbf u \right\}
$$
represents the minimal energy required to go from a constant state $\mf v$ to another constant state $\mf u$.  The notation $\textrm{Per}_\Omega(\mathbf u=\mathbf a)$ is the perimeter of the set $\{\mf u = \mf a\}$ in $\Omega$, and $\mathcal H^2$ represents the two-dimensional Hausdorff measure.
 
The intuition why $\ml{G}_0$ contains an interfacial energy can be explained as the follows.   It is expected that the minimizer $\mf{u}_\ep$ of $\ml{E}_\ep$  has a sharp transition from state $\mf{a}$ to state $\mf{b}$ across the interface $\p U\cap\Omega$, but have {no variation} up to order $\ep$ along the tangential direction of the interface.   The layer thickness should be of order $\ep$ so that the kinetic energy $\ep \int |\grad \mf{u}|^2 $and the bulk energy $\frac{1}{\ep}\int W(u)$ have the same order of magnitude.   The minimal energy occurs only when these two energies are balanced, that is
\[
\int_\Omega \ep |\grad \mf{u} |^2 + \frac{1}{\ep}W(\mf{u}) = 2 \int_\Omega \sqrt{W(\mf{u})} |\grad \mf{u}|\, dx. 
\]
By the co-area formula, the energy contributed by the internal interface is roughly  $2g(\mf{a},\mf{b}) \textrm{Per}_\Omega(\mathbf u=\mathbf a)$.
This is the interfacial energy.  Similar argument can also explain the appearance of  the boundary layer energy in $\ml{G}_0$.

Finally, the physical solution is selected by
\[
\inf_{\mf{u}\in \ms{A}_0} \ml{G}_0[\mf{u}].
\]
This minimization problem is a geometric problem and can be solved by standard direct method in calculus of variations.

\subsection{Main Theorems}
We list our main theorems below.  Although their proofs are mainly followed by the standard procedure of $\Gamma$-convergence arguments in \cite{S,S1}, the quadratic constraints (i.e. $\ml{N}[\mf{u}]=N$, $\ml{M}[\mf{u}]=M$) in our present study require some modifications. We put these proofs in the next section for completeness.

\begin{thm} \label{Thm:Gamma-limit}
The sequence $\{\ml G_\epsilon\}$ $\Gamma$-converges to $\ml G_0$ in $L^2(\Omega)$-topology. This means the follows:
\begin{enumerate}
 \item (Lower semi-continuity) For any sequence $\{\mathbf u_\epsilon\}\subset \ms{A}$ converging to some $\mathbf u_0\in\ms{A}_0$ in $(L^2(\Omega))^3$, we have
        $$ \ml G_0[\mathbf u_0] \le \liminf_{\epsilon\to 0} \ml G_\epsilon[\mathbf u_\epsilon].$$
  \item (Recovery sequence) For any $\mathbf v_0\in \ml{A}_0$, there exists a sequence $\{\mathbf v_\epsilon\}\subset \ms{A}$  converging to $\mathbf v_0$ in $(L^2(\Omega))^3$  such that
  $$ \ml G_0[\mathbf v_0] = \lim_{\epsilon\to 0} \ml G_\epsilon [\mathbf v_\epsilon].$$
\end{enumerate}
\end{thm}
\noindent Such $\ml{G}_0$ is called the $\Gamma$-limit of $\ml{G}_\ep$.

\begin{thm}\label{Thm:compactness}
Suppose that $\{\mathbf u_\epsilon\}$  is a family in $\ms{A}$ with an uniformly bounded energy, that is
\be 
\ml{G}_\epsilon[\mathbf u_\epsilon]\le C_0 \label{eq:boundedness}
\ee
for some positive constant $C_0$.
Then there exists a subsequence $\{\mathbf u_{\epsilon_j}\}$ converges to some $\mathbf u_0\in\ms{A}_0$ in $(L^2(\Omega))^3$ as $j\to\infty$. 
\end{thm}

\begin{thm}\label{thm:5.5}
 Suppose $\mathbf u_\epsilon$ are  minimizers of the variational problem
 $$ \inf_{\mathbf u\in \ms{A}} \ml G_\epsilon[\mathbf u]$$
 and $\{\mathbf u_{\epsilon_j}\}$ converges $\mathbf u_0\in\ms{A}_0$ in $(L^2(\Omega))^3$ for some subsequence $\epsilon_j\to0$. 
 Then $\mathbf u_0$ solves the variational problem
 $$ \inf_{\mathbf u\in \ms{A}_0
 }\ml G_0[\mathbf u] .$$
\end{thm}
\begin{proof}
 Let $\mathbf v_0\in \ms{A}_0$. There exists a sequence $\mathbf v_\epsilon\in \mathcal A$ such that
 $\ml G_0[\mathbf v_0]=\lim_{\epsilon\to 0} \ml G_\epsilon[\mathbf v_\epsilon] $. We have
 $$ \ml G_0[\mathbf v_0] = \lim_{\epsilon\to 0} \ml G_\epsilon[\mathbf v_\epsilon] \ge \liminf_{\epsilon\to0}\ml G_\epsilon[\mathbf u_\epsilon)]\ge \ml G_0[\mathbf u_0].$$
 This shows that $\mathbf u_0$ minimizes the functional $\ml G_0$.
\end{proof}

According to Theorem \ref{thm:5.5} and the discussion in sections \ref{case:NS-MS}, \ref{case:2C}, \ref{case:NS-2C}, \ref{case:MS-MS} and \ref{case:3C},  we conclude that the asymptotic behaviors of the ground states of the Spin-1 BEC systems are characterized by their corresponding Thomas-Fermi solutions satisfying the minimal interface criterion.

By the work of \cite{KS}, it is also possible to construct local minimizers of the perturbed variational problem from isolated local minimizers of the limiting one.
\begin{thm}
 Suppose $\mathbf u_0\in \ms{A}_0$ is an isolated $L^2$-local minimizer of the functional $\ml G_0$, i.e. there exists $\epsilon>0$ such that 
 $$ \ml G_0(\mathbf v) > \ml G_0(\mathbf u_0) \qquad\textrm{ for all } \mf{v}\in\ml{A}_0\textrm{ and } 0<\|\mathbf v - \mathbf u_0\|_{L^2}<\epsilon.$$
  Then there exists a sequence $\{\mathbf u_\epsilon\}\subset \ms{A}$ such that each $\mathbf u_\epsilon$ is a local minimizer of the functional $\ml G_\epsilon$ and $\mathbf u_\epsilon\to\mathbf u_0$ in $(L^2(\Omega))^3$ as $\epsilon\to0$.
\end{thm}


\begin{figure}[ht]
\begin{center}
\includegraphics[scale=0.5]{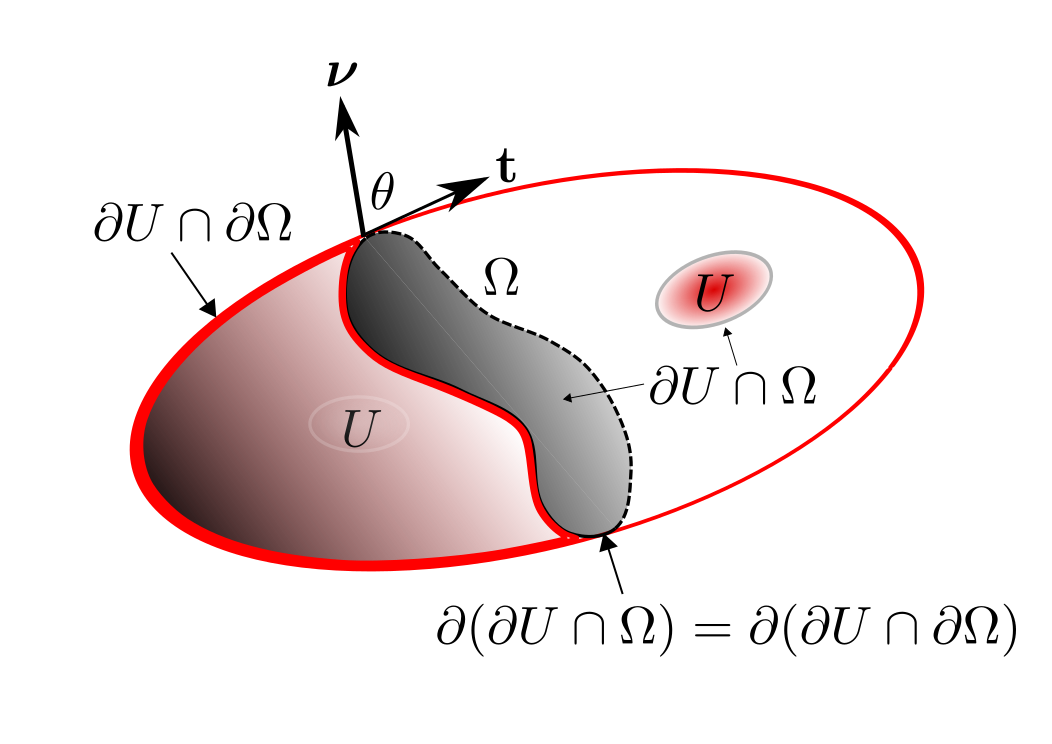}
\caption{The geometry of the domain $\Omega$ and the interface $\partial U\cap \Omega$}
\label{fig:3domain}
\end{center}
\end{figure}

 The existence of a minimizer for the limiting problem
 $$\inf_{\mathbf u\in \ms{A}_0} \ml G_0[\mathbf u]$$
 is obtained through the standard direct method in the calculus of variations.  By the straight forward calculation of the first variation, we get  the following necessary condition  for the interface. 

 \begin{thm}
  Let $\Omega$ be a bounded domain with a $C^2$-boundary $\partial\Omega$ and $\mathbf u_0$ be a critical point of $\ml G_0$ such that $\partial U\cap\Omega$ is of class $\mathcal C^2$ with mean curvature $H:\partial U\cap\Omega\to \mathbb R$. 
  The corresponding Euler-Lagrange equation of $\min \ml{G}_0[\mf{u}]$ is
  \begin{equation*}
   \left\{\begin{array}{ll}
     H(x) = \textrm{Const.} &\qquad \textrm{for }x\in \partial U\cap\Omega,\\
    g(\mathbf a,\mathbf b)(\boldsymbol\nu\cdot \mathbf t) + g(\mathbf 0,\mathbf a) - g(\mathbf 0,\mathbf b) = 0 & \qquad\textrm{for }x\in \partial (\partial U\cap\Omega)
   \end{array}\right.
  \end{equation*}
   where  $\boldsymbol \nu:\partial (\partial U\cap\Omega)\to \mathbb S^{n-1}$ is an outward unit tangential vector to the interface $\partial U\cap\Omega$ and normal to $\partial (\partial U\cap\Omega)$; $\mathbf t:\partial(\partial U\cap \Omega)\to\mathbb S^{n-1}$ is the outward unit tangential vector to $\partial U\cap \partial\Omega$ and normal to $\partial(\partial U\cap\partial \Omega)$. 
 \end{thm} 
 
\begin{proof}
The calculation the  first variation of $\ml{G}_0$ under the fixed volume constraint is similar to  \cite{CS,Shieh}. We omit it here.
\end{proof} 

\begin{remark}
The second equation of the Euler-Lagrange equation is called the Young's relation, which appears in natural process of wetting: 
$$ g(\mathbf a,\mathbf b)\cos\theta + g(\mathbf 0,\mathbf a) - g(\mathbf 0,\mathbf b) = 0.$$
Here, the contact angle between the boundary and the interface is denoted by $\theta$, see Figure \ref{fig:3domain}.
Indeed, it is a balance law between surface tensions of three different interfaces on the boundary.  Thus, the Euler-Lagrange equation mentioned above could also be considered as a equation for a quantum-like wetting process.
\end{remark}


\subsection{Preliminary lemmas}
The function $W:\mathbb R^3\to\mathbb R$  constructed in those sections has  the following properties:
\begin{enumerate}
 \item $W$ is a $C^1$-nonnegative function with the  following symmetry property:
\be
W(u_1,u_0,u_{-1})=W(-u_1,u_0,u_{-1})=W(u_1,-u_0,u_{-1})=W(u_1,u_0,-u_{-1}) \label{cond:W1}
\ee
 and
 $$ W(\mathbf u)=0  \textrm{ in } \mathbb R^3_+  \ \textrm{ if and only if } \mathbf u=\{\mathbf a, \mathbf b\} .$$
 \item There exist $\delta>0$ and $C>0$ such that
 \be C|\mathbf u-\mathbf a|^2 \le W(\mathbf u) \le \frac{1}{C} |\mathbf u-\mathbf a|^2 \quad\textrm{ for }|\mathbf u-\mathbf a|<\delta \label{cond:W2.1}
 \ee
 and
 \be
 C|\mathbf u-\mathbf b|^2 \le W(\mathbf u) \le \frac{1}{C} |\mathbf u-\mathbf b|^2 \quad \textrm{ for }|\mathbf u-\mathbf b|<\delta.\label{cond:W2.2}
 \ee
 \item There exist two positive values $C$ and $R$ such that
  \be  C|\mathbf u|^2 \le W(\mathbf u)\qquad\textrm{for }|\mathbf u|>R .\label{cond:W3}
  \ee
\end{enumerate}
We shall assume these properties of $W$ in the discussion below.

We quote  several useful lemmas from \cite{Baldo,M,S,S1}   which will be used in the proof of our $\Gamma$-convergence result.

\begin{lem}[See \cite{M,S}]\label{lemma:4.1}
 Let $\Omega$ be an open bounded subset in $\mathbb R^n$ with Lipschitz-continuous boundary. Let $A$ be an open subset in $\mathbb R^n$ with compact  $C^2$-boundary $\partial A$ such that $H^{n-1}(\partial A\cap \partial \Omega)=0$.
 Define the signed distance function to $\partial A$, $d:\Omega\to \mathbb R$, by
\be \label{eq:dA}
 d_A(x)=\left\{ \begin{array}{rl}
            \textrm{dist}(x,\partial A) & x\in \Omega\backslash A,\\
            -\textrm{dist}(x,\partial A) & x\in A\cap \Omega.
          \end{array}  \right.
 \ee
 Then for some $s>0$, $d$ is $C^2$ function in $\{|d(x)|<s\}$ with
 $$ |\nabla d_A| = 1.$$
 Furthermore,
 \be  \lim_{s\to 0 } \mathcal H^{n-1}(\{ x\in\Omega:\,d_A(x)=s \})= \mathcal H^{n-1}(\partial A\cap\Omega). \label{eq:measure}\ee
\end{lem}

\begin{lem}[See \cite{Baldo,S,S1}]\label{lemma:4.2}
For any $\boldsymbol \xi_0,\boldsymbol\xi_1\in \mathbb R^3_+$, there exists a curve $\boldsymbol\gamma_{\boldsymbol\xi_0,\boldsymbol\xi_1}:[0,1]\to\mathbb R^3_+$ with $\boldsymbol\gamma(0)=\boldsymbol\xi_0$ and $\boldsymbol\gamma(1)=\boldsymbol\xi_1$ which minimizes the degenerate geodesic problem
$$ g(\boldsymbol\xi_0,\boldsymbol\xi_1):=\inf
     \left\{ \int_0^1 \sqrt{W(\boldsymbol\gamma(\tau))}\,|\boldsymbol\gamma'(\tau)|\,d\tau:\, \boldsymbol \gamma \in Lip([0,1];\mathbf R^3_+),\, \boldsymbol \gamma(0)=\boldsymbol\xi_0,\,\boldsymbol\gamma(1)=\boldsymbol\xi_1 \right\}.$$
 Define $\varphi_{\boldsymbol \xi_0}(\boldsymbol\xi_1)=g(\boldsymbol\xi_0,\boldsymbol\xi_1)$.
 Then the  function
 $\varphi_{\boldsymbol\xi_0}:\mathbb R^3_+\to\mathbb R_+$ is Lipchitz-continuous with the property
 $$|\nabla \varphi_{\boldsymbol\xi_0}(\boldsymbol\xi_1)|=\sqrt{W(\boldsymbol\xi_1)} \quad \textrm{ for almost all } \boldsymbol\xi_1\in \mathbb R^3_+. $$
 Furthermore, if $\mathbf u\in H^1(\Omega;\mathbb R^3_+)\cap L^\infty(\Omega;\mathbb R^3_+)$,  then $\varphi_{\boldsymbol\xi_0}\circ\mathbf u\in W^{1,1}(\Omega;\mathbb R^3_+)$ and
 \begin{equation}
 \int_\Omega |\nabla(\varphi_{\boldsymbol\xi_0}\circ\mathbf u)|\,dx \le \int_\Omega \sqrt{W(\mathbf u(x))}|\nabla \mathbf u(x)|\,dx, \label{eq:4.6}
 \end{equation}
 where $|\nabla\mathbf u|$ is the $2$-norm of  $\nabla\mathbf u$, that is
 $$ |\nabla \mathbf u| = \left( \sum_{i,j=1}^3 \left(\frac{\partial u_j}{\partial x_i}\right)^2\right)^{1/2}.$$
\end{lem}

\begin{proof}
\bn
\I This geodesic problem corresponds to a Riemannian metric $ds^{2} = W ds_{Euc}^{2}$ on $\R_{+}^{3}$, where $ds_{Euc}$ is the Euclidean matric.     However, there are two difficulties here. The first one is that this is a constrained variational problem, namely all components of the geodesic curve $\ms{\gamma}(\cdot)$ should be non-negative.
To resolve this difficulty, we extend this variational problem to the {\em entire} $\mathbb R^3$ by taking the advantage of the symmetry property of the function $W$.  Thus we  consider an equivalent degenerate geodesic problem on  $\mathbb R^3$
\begin{equation}\label{eq:gcurve0}
\inf
     \left\{ \int_0^1 \sqrt{W(\boldsymbol\gamma(\tau))}\,|\boldsymbol\gamma'(\tau)|\,d\tau:\, \boldsymbol \gamma \in Lip([0,1];\mathbb R^3),\, \boldsymbol \gamma(0)=\boldsymbol\xi_0,\,\boldsymbol\gamma(1)=\boldsymbol\xi_1 \right\},
\end{equation}
For any $\boldsymbol \xi_0,\boldsymbol\xi_1\in \mathbb R^3_+$, if we have found a Lipschitz geodesic  $\ms{\gamma}(\cdot)$ connecting $\bs{\xi}_{0}$ to $\bs{\xi}_{1}$ in $\R^{3}$, then using the symmetry property of $W$ and reflection, we can always find a representative Lipschitz-continuous curve $\boldsymbol\gamma:[0,1]\to\mathbb R^3_+$ which solves the constrained geodesic problem. 
\I The second difficuly is the degeneracy of $W(u)$ at $\mf{a}$ and $\mf{b}$.
Thus, the direct method in the calculus of variations cannot be applied straightforwardly.  Therefore, we consider the regularized problem:
\begin{equation}\label{eq:gcurve1}
 \inf
     \left\{ \int_0^1 \left(\sqrt{W(\boldsymbol\gamma(\tau))+ \delta}  \right)\,|\boldsymbol\gamma'(\tau)|\,d\tau:\, \boldsymbol \gamma \in Lip([0,1];\mathbb R^3),\, \boldsymbol \gamma(0)=\boldsymbol\xi_0,\,\boldsymbol\gamma(1)=\boldsymbol\xi_1 \right\},
\end{equation}     
 where $\delta > 0$.   The regularized Riemannian metric $ds_{\delta}^{2} = (W+\delta)ds_{Euc}^{2}$ 
 is conformal to the Euclidean metric on the plane.
 Its minimizer $\boldsymbol{\gamma}_\delta$ exists uniquely by the direct method in the calculus of variations.
A uniform bound on the Euclidean arclength $\ell_\delta$ of  $\boldsymbol\gamma_\delta$ can be derived by  carefully analyzing the curve $\boldsymbol\gamma_\delta$ in two different regions: one region is away from the zeros of $W$ and the other is the region near the zeros of $W$. Because the value $\ell_\delta$  is invariant under re-parametrization of the curve $\boldsymbol\gamma_\delta$, this allows us to choose a new parametrization $\tilde{\boldsymbol\gamma}_\delta:[0,1]\to\mathbb R^3$ with a constant speed, that is $|\tilde{\boldsymbol\gamma}_\delta'|\equiv \ell_\delta$.
  By the Arzel\'a-Ascoli compactness theorem, there exists a subsequence  $\{\tilde{\boldsymbol\gamma}_{\delta_k}\}$ and a curve $\boldsymbol\gamma$ such that 
 $\tilde{\boldsymbol\gamma}_{\delta_k}$ converges to $\boldsymbol\gamma$ in $C([0,1],\R^3)$. It is easy to see that  $\boldsymbol\gamma$ solves the problem \eqref{eq:gcurve0}.
\I
It is observed that $g:\R^3_{+}\times\R^3_{+}\to\R$ is a metric in $\R^3_{+}$ and  satisfies the triangle inequality:
$$ g(\boldsymbol\xi_0,\boldsymbol\xi_1)\le g(\boldsymbol\xi_0,\boldsymbol\xi_1')+ g(\boldsymbol\xi_1',\boldsymbol\xi_1).$$
That is,
\be
\label{eq:geoineq}
 \varphi_{\boldsymbol\xi_0}(\boldsymbol\xi_1)-\varphi_{\boldsymbol\xi_0}(\boldsymbol\xi_1')\le g(\boldsymbol\xi_1,\boldsymbol\xi_1')= \int_0^1\sqrt{W(\boldsymbol\gamma_{\boldsymbol\xi_1',\boldsymbol\xi_1}(\tau))}|\boldsymbol\gamma_{\boldsymbol\xi_1',\boldsymbol\xi_1}'(\tau)|\,d\tau.
 \ee
From this, we get that  $\varphi_{\boldsymbol\xi_0}$ is locally Lipchitz continuous and
\be \label{eq:geoineq2}
  |\nabla\varphi_{\boldsymbol\xi_0}(\boldsymbol\xi_1)| \le \sqrt{W(\boldsymbol\xi_1)}
 \ee
 As we choose $\boldsymbol\xi_1'$ moving along the geodesic $\boldsymbol\gamma_{\boldsymbol\xi_0,\boldsymbol\xi_1}$ to approach
$\boldsymbol\xi_1$, then the inequality in (\ref{eq:geoineq}) becomes equality, and we get
\be \label{eq:geoeq2}
 |\nabla\varphi_{\boldsymbol\xi_0}(\boldsymbol\xi_1)| = \sqrt{W(\boldsymbol\xi_1)}.
 \ee
 Notice that the above arguments hold for all $\bs{\xi}_{1}$ in the interior of $\R_{+}^{3}$.  But the inequality (\ref{eq:geoineq2}) and equality (\ref{eq:geoeq2}) can be extended to the boundary of $\R_{+}^{3}$.
 \I  Applying the Cauchy inequality and from (\ref{eq:geoineq2}), we get
 $$
 \aligned
   \int_\Omega |\nabla(\varphi_{\boldsymbol\xi_0}\circ\mathbf u)|\,dx
    =& \int_\Omega  \left| \left(\nabla \varphi_{\boldsymbol\xi_0}(\mathbf u) \right)\nabla_x\mathbf u\right|\,dx\\
    \le& \int_\Omega  \left| \nabla \varphi_{\boldsymbol\xi_0}(\mathbf u) \right|\left|\nabla_x\mathbf u\right|\,dx
 \le& \int_\Omega \sqrt{W(\mathbf u(x))}|\nabla \mathbf u(x)|\,dx
 \endaligned
 $$
  for $\mathbf u\in C^1(\Omega;\mathbb R^3_+)$.  By the density theorem, this  inequality  also holds for $\mathbf u\in H^1(\Omega;\mathbb R_+^3)\cap L^\infty(\Omega;\mathbb R_+^3)$.
 \en
\end{proof}

The following lemma is used to construct the one-dimensional profile of the internal layer.
\begin{lem}[See \cite{S,S1}]\label{lemma:4.3}
 Given $\mf{a},\mf{b} \in \R_+^3$. Then there exists a  Lipchitz-continuous function $\boldsymbol \eta:(-\infty,\infty)\to\mathbb R^3_+$  whose trajectory is the geodesic with the metric $ds^{2} = Wds^{2}_{Euc}$  connecting $\mf{a}$ to $\mf{b}$. The function $\bs{\eta}(\cdot)$ solves the variational problem
$$
 \inf_{\substack{ \boldsymbol \eta(-\infty)=\mathbf a\\ \boldsymbol\eta(+\infty)=\mathbf  b}}
   \int_{-\infty}^\infty |\boldsymbol\eta'|^2 + W(\boldsymbol\eta)\,ds
 $$
with minimal value $2g(\mathbf a,\mathbf b)$ and
\be |\boldsymbol\eta(t)-\mathbf a|<Ce^{c_1t}\qquad\textrm{as }t\to-\infty,\label{eq:ex-estimate1}\ee
\be |\boldsymbol\eta(t)-\mathbf b|<Ce^{-c_1t}\qquad\textrm{as }t\to+\infty.\label{eq:ex-estimate2}\ee
\end{lem}

\begin{proof}
Let $\bs{\gamma}(\cdot)$ be the geodesic curve  in Lemma \ref{lemma:4.2} connecting $\mf{a}$ to $\mf{b}$.  Let us parametrize it by $\beta\in [0,1]$ such that $\bs{\gamma}(0)=\mf{a}$, $\bs{\gamma}(1)=\mf{b}$, $\bs{\gamma}(\cdot)$ is Lipschitz continuous  and $|\bs{\gamma}'(\beta) |\ge c > 0$ for some constant $c$ and for all $\beta \in [0,1]$. 
Now, let us consider  the ODE:
$$
\left\{ \begin{array}{rcl}
 \frac{d\beta}{dt}(t) &=& \frac{\sqrt{W(\boldsymbol\gamma (\beta) ) } }{|\boldsymbol\gamma'(\beta)|} ,\\
 \beta(0) &=& \frac{1}{2} .
 \end{array}
\right.
$$
The right-hand side of this ODE, $f(\beta):= \sqrt{W(\boldsymbol\gamma(\beta))} /|\boldsymbol \gamma'(\beta)| $, is Lipchiptz continuous on $[0,1]$, thus we have local existence and uniqueness of the solution.  Note that  $\beta =0, 1$ are the only two zeros of $f$, thus, from uniqueness, the solution of this ODE stays between $0$ and $1$, as long as it exists. Thereby it exists globally.   Furthermore, 
$\beta(t)\to 0$ (resp. $\beta(t)\to 1$) exponentially fast as $t\to -\infty$ (resp. $t\to \infty$) because $f'(0)$ (resp. $f'(1)\ne 0$), which is, in turn, due to (\ref{cond:W2.1}) (resp. (\ref{cond:W2.2})).   

Now let us define $\boldsymbol\eta(t):=\boldsymbol\gamma(\beta(t))$.  We have
\[
|\bs{\eta}'(t)| = |\bs{\gamma}'(\beta(t))\beta'(t)| =  |\bs{\gamma}'(\beta(t))|\beta'(t)=\sqrt{W(\bs{\gamma}(\beta(t)))}=\sqrt{W(\bs{\eta}(t))}.
\]
Let $\hat{\boldsymbol\eta}\in Lip ((-\infty,\infty), \R_{+}^{3})$ be any curve connecting $\mf{a}$ to $\mf{b}$.   Using the Cauchy inequality,  the fact that  $\boldsymbol\gamma$ is  geodesic and $|\bs{\eta}'|=W(\bs{\eta})$,  we get
$$ \int_{-\infty}^\infty |\hat{\boldsymbol\eta}'|^2 + W(\hat{\boldsymbol\eta})\,dt \ge 2 \int_{-\infty}^\infty \sqrt{  W(\hat{\boldsymbol\eta}) }{  |\hat{\boldsymbol\eta}'| }\,dt
\ge  2 \int_{-\infty}^\infty \sqrt{  W(\boldsymbol\eta) }{  |\boldsymbol\eta'| }\,dt
=\int_{-\infty}^\infty |\boldsymbol\eta'|^2 + W(\boldsymbol\eta)\,dt
$$
Thus, $\boldsymbol\eta$ solves the variational problem.
\end{proof}

Similarly, we also have the lemma for the construction of the one-dimensional profile of the boundary layer due to the difference between the Dirichlet boundary condition $\mathbf u=\mathbf 0$ on $\partial \Omega$ and $\mathbf u=\mathbf a,\,\mathbf b$ in $\Omega$.

\begin{lem}\label{lemma:4.9}
 There exist two  Lipchitz-continuous functions $\boldsymbol \eta_{\mathbf a}:(0,\infty)\to\mathbb R^3_+$  and $\boldsymbol \eta_{\mathbf b}:(0,\infty)\to\mathbb R^3_+$, whose trajectories are  the geodesics with the metric $ds^2 = W ds^2_{Euc}$ connecting $\mf 0$ to $\mf{a}$ and to $\mf b$, respectively. 
They solve the variational problems
$$
 \inf_{\substack{ \boldsymbol \eta(0)=\mathbf 0\\ \boldsymbol\eta(+\infty)=\mathbf  a}}
   \int_0^\infty |\boldsymbol\eta'|^2 + W(\boldsymbol\eta)\,dt
 $$
 and
 $$
 \inf_{\substack{ \boldsymbol \eta(0)=\mathbf 0\\ \boldsymbol\eta(+\infty)=\mathbf  b}}
   \int_0^\infty |\boldsymbol\eta'|^2 + W(\boldsymbol\eta)\,dt,
 $$
 with minimal values
  $2g(\mathbf 0,\mathbf a)$ and $2g(\mathbf 0,\mathbf b)$, respectively.
Furthermore, $\lim_{t\to+\infty}\boldsymbol\eta_{\mathbf a}(t)=\mathbf a$ and $\lim_{t\to+\infty}\boldsymbol\eta_{\mathbf b}(t)=\mathbf b$ are being attained at  exponential rates.
\end{lem}
\begin{proof}
The proof is similar to the proof of Lemma \ref{lemma:4.3}.
\end{proof}

\subsection{Proofs of the Main Theorems}
\paragraph{Proof of Theorem \ref{Thm:Gamma-limit}}
\begin{flushleft}{\em Proof of Lower semi-continuity.}\end{flushleft}
\noindent 
1. Let us extend  $\mf{u}_\ep$ and $\mf{u}_0$ trivially to a larger bounded smooth domain $\Omega'$ 
such that  $\overline{\Omega}\subset \subset \Omega'$.  That is, 
$$ \tilde{\mathbf u}_\epsilon(x) = \left\{\begin{array}{ll}
 \mathbf u_\epsilon(x)&\quad\textrm{for }x\in \Omega,\\
 \mathbf 0 &\quad\textrm{for }x\in \Omega'\backslash\Omega,
\end{array}\right.
\qquad\textrm{and}\qquad
\tilde{\mathbf u}_0 (x)= \left\{\begin{array}{ll}
 \mathbf u_0(x)&\quad\textrm{for }x\in \Omega,\\
 \mathbf 0 &\quad\textrm{for }x\in \Omega'\backslash\Omega.
\end{array}\right.
$$
We have $\tilde{\mf{u}}_\epsilon \to \tilde{\mf{u}}_0$ in $(L^2(\Omega'))^3$.  This together with the fact that the two functions $\varphi_{\mathbf a}$ and $\varphi_{\mathbf 0}$ are Lipschitz continuous ( Lemma \ref{lemma:4.2}) lead to
 $$
(\varphi_{\mathbf a}\circ\tilde{\mathbf u}_\epsilon)(x) \to  (\varphi_{\mathbf a}\circ\tilde{\mathbf u}_0)(x)=\left\{\begin{array}{ll}
       0 &\quad x\in U \\
       g(\mathbf a,\mathbf b) &\quad x\in\Omega\backslash U \\
       g(\mathbf a,\mathbf 0) &\quad x\in\Omega'\backslash\Omega
    \end{array}\right.\quad \textrm{in }L^2(\Omega')
$$
and
$$
(\varphi_{\mathbf 0}\circ\tilde{\mathbf u}_\epsilon)(x) \to  (\varphi_{\mathbf 0}\circ\tilde{\mathbf u}_0)(x)=\left\{\begin{array}{ll}
       g(\mathbf 0,\mathbf a) &\quad x\in U \\
       g(\mathbf 0,\mathbf b) &\quad x\in\Omega\backslash U \\
       0 &\quad x\in\Omega'\backslash\Omega
    \end{array}\right.\quad \textrm{in }L^2(\Omega').
$$

\noindent
2. Let $\Lambda_1^\delta$ and $\Lambda_2^\delta$ be two bounded sets defined by
$$
\Lambda_1^\delta = \{x\in\Omega: \textrm{dist}(x,\partial\Omega)\ge \delta\} \quad\textrm{and}\quad
\Lambda_2^\delta = \{x\in \Omega: \textrm{dist}(x,\partial\Omega) < \delta\} \cup (\Omega'\backslash \Omega).
$$
By using the inequality of arithmetic and geometric means and the lower semicontinuity of the BV-norm under $L^1$-convergence, we have
\begin{align*}
 &\liminf_{\epsilon\to 0}\ml G_\epsilon[\mathbf u_\epsilon] \ge 2\liminf_{\epsilon\to 0}\int_\Omega \sqrt{W(\mathbf u_\epsilon (x))}\,|\nabla \mathbf u_\epsilon(x)|\,dx \\
 &= 2\liminf_{\epsilon\to 0} \left( \int_{\Lambda_1^\delta}  \sqrt{W(\tilde{\mathbf u}_\epsilon (x))}\,|\nabla \tilde{\mathbf u}_\epsilon(x)|\,dx
 +\int_{\Lambda_2^\delta}  \sqrt{W(\tilde{\mathbf u}_\epsilon (x))}\,|\nabla \tilde{\mathbf u}_\epsilon(x)|\,dx\right)\\
 &\ge 2\liminf_{\epsilon\to 0}\int_{\Lambda_1^\delta}  \sqrt{W(\tilde{\mathbf u}_\epsilon (x))}\,|\nabla \tilde{\mathbf u}_\epsilon(x)|\,dx
 + 2\liminf_{\epsilon\to 0}\int_{\Lambda_2^\delta}  \sqrt{W(\tilde{\mathbf u}_\epsilon (x))}\,|\nabla \tilde{\mathbf u}_\epsilon(x)|\,dx
\end{align*}
After applying the inequality \eqref{eq:4.6} to the function $\varphi_{\mathbf a}\circ\tilde{\mf{u}}_\epsilon$ on the domain $\Lambda_1^\delta$ and the function $\varphi_{\mathbf 0}\circ\tilde{\mf{u}}_\epsilon$ on the domain $\Lambda_2^\delta$, we find that
\begin{align*}
\liminf_{\epsilon\to 0}\ml G_\epsilon[\mathbf u_\epsilon]
&\ge 2 \liminf_{\epsilon\to 0}
 \int_{\Lambda_1^\delta} | \nabla(\varphi_{\mathbf a}\circ\tilde{\mathbf u}_\epsilon)|\,dx
 +2 \liminf_{\epsilon\to 0}
 \int_{\Lambda_2^\delta} | \nabla(\varphi_{\mathbf 0}\circ\tilde{\mathbf u}_\epsilon)|\,dx\\
 &\ge 2 \left(
  \int_{\Lambda_1^\delta} | \nabla(\varphi_{\mathbf a}\circ\tilde{\mathbf u}_0)|
  +\int_{\Lambda_2^\delta} | \nabla(\varphi_{\mathbf 0}\circ\tilde{\mathbf u}_0)|
  \right).
\end{align*}
Taking $\delta\to 0$, we get
\begin{align*}
\liminf_{\epsilon\to 0}\ml G_\epsilon[\mathbf u_\epsilon]  &\ge  2 g(\mathbf a,\mathbf b) \,\textrm{Per}_\Omega(\mathbf u=\mathbf a)
+ 2 g(\mathbf 0,\mathbf a)\,\mathcal H^2(\{x\in\partial\Omega:\mathbf u=a\})\\
& \qquad
+ 2 g(\mathbf 0,\mathbf b)\,\mathcal H^2(\{x\in\partial\Omega:\mathbf u=b\}).
\end{align*}
\begin{flushright}
{$\Box$}
\end{flushright}

\noindent
{\em Proof of Recovery sequence.} \\
1.  We will construct recovery sequence $\{\mf{v}_\ep\}\subset \ms{A}$ for 
 $\mathbf v_0\in \ms{A}_0$ (\ref{eq:A0}) of the form
 \be
\mathbf v_0 =\left\{\begin{array}{ll}
           \mathbf a &\qquad\textrm{if }x\in V, \\
           \mathbf b &\qquad\textrm{if }x\in \Omega\backslash V,
        \end{array}\right. \label{eq:v0}
\ee
and satisfies $\ml{N}[\mf{v}_0]=N, \ml{M}[\mf{v}_0]=M$. 
We discuss the case when
 $\mf{v}_0\in \ms{A}_0\cap (BV(\Omega))^3$ first. The case when  $\mathbf v_0\in \ms{A}_0\setminus (BV(\Omega))^3$ will be discussed in item 6 below.
In the first case, since $\mf v_0 \in (BV(\Omega))^3$, the set $V$ has finite perimeter in $\Omega$. We may assume that the interface $\Gamma = \partial V \cap\Omega$ is smooth because a set of finite perimeter can be approximated by a sequence of sets with smooth boundary, see \cite{S}.

\noindent
2. The recovery sequence $\mf{v}_\ep$ to be constructed will have the form
\[
\mathbf v_\epsilon(x) = \mf{w}_\epsilon(x) + \alpha_1(\epsilon)\, \boldsymbol\varphi(x) + \alpha_2(\epsilon) \,\boldsymbol\psi(x).
\]
Here, $\mf{w}_\ep$ is a layer solution; $\boldsymbol\varphi(x) $, $\boldsymbol\psi(x)$ are smooth functions supported on $V$ and $\Omega \setminus V$, respectively; the coefficients $\alpha_1(\ep), \alpha_2(\ep)$ are $O(\ep)$.  The terms  $\alpha_1(\epsilon)\, \boldsymbol\varphi(x) + \alpha_2(\epsilon) \,\boldsymbol\psi(x)$ are designed so that the conservation constraints $\ml{N}[\mf{v}_\ep]=N$ and $\ml{M}[\mf{v}_\ep]=M$ are satisfied. Detail conditions for $\boldsymbol\varphi(x) $ and $\boldsymbol\psi(x)$ will be given later.
The layer solution $\mf{w}_\ep$ involves an internal layer function $\boldsymbol{\eta}$ and an auxiliary boundary layer function $\tilde{\boldsymbol\eta}$.  
The internal layer function $\boldsymbol{\eta}$  is defined in Lemma \ref{lemma:4.3}, which connects  $\mf{a}$ to  $\mf{b}$.  The auxiliary boundary layer function 
$\tilde{\boldsymbol\eta}:(-\infty,\infty)\to\mathbb R^3_+$  is defined by
$$
 \tilde{\boldsymbol\eta} (t) = \left\{\begin{array}{ll}
  \boldsymbol\eta_{\mathbf a}(-t) & \qquad\textrm{if }t<0, \\
  \boldsymbol\eta_{\mathbf b}(t) & \qquad \textrm{if }t\ge 0,
 \end{array}\right.
$$
where $\boldsymbol\eta_{\mathbf a}$ and $\boldsymbol\eta_{\mathbf b}$ are defined in Lemma \ref{lemma:4.9}. They are the boundary layer functions connecting  $\mf{0}$ to $\mf{a}$  and to $\mf{b}$, respectively. 
Note that the function $\tilde{\boldsymbol\eta}(t)$ on the negative (resp. positive) half line $\{t< 0\}$ (resp. $ \{ t > 0\}$) is the profile of the boundary layer approaching to the value $\mathbf a$ (resp. $\mathbf b)$. 
Let us also define the signed distance function $d=d_V$ associated with $\p V$  by (\ref{eq:dA}) and an 
 auxiliary distance function associated with $\p \Omega$ by
\[
d_b(x) = \left\{\begin{array}{ll}
-\textrm{dist}(x,\p \Omega) & \mbox{ if } x \in V, \\
\textrm{dist}(x,\p \Omega) & \mbox{ if } x \in \Omega \setminus V.
\end{array}\right.
\]
We choose a cut-off function  $\zeta$, which is  a smooth function such that $0\le \zeta\le 1$ and
$$
\zeta(t) =\left\{\begin{array}{ll}
      1 &\qquad\textrm{if }|t|\le 1,\\
      0&\qquad\textrm{if }|t|\ge 2.
     \end{array}\right.
$$
Let us set $\gamma=\frac{2}{3}$.
Finally, we define the layer solution 
 $\mf{w}_\epsilon$ by
\begin{align*}
 \mf{w}_\epsilon (x) &= \left[1-\zeta\left(\frac{d_b(x)}{\epsilon^\gamma}\right)\right]
     \left\{ \zeta\left(\frac{d(x)}{\epsilon^\gamma}\right)\boldsymbol\eta\left(\frac{d(x)}{\epsilon}\right) + \left[1-\zeta\left(\frac{d(x)}{\epsilon^\gamma}\right)\right]\mathbf v_0(x) \right\}  \\
     & \qquad +  \zeta\left(\frac{d_b(x)}{\epsilon^\gamma}\right)
     \left[1- \zeta\left(\frac{d(x)}{\epsilon^\gamma}\right) \right]
     \tilde{\boldsymbol \eta}\left(\frac{d_b(x)}{\epsilon}\right).
\end{align*}
\begin{figure}[ht]
\begin{center}
\includegraphics[scale=0.5]{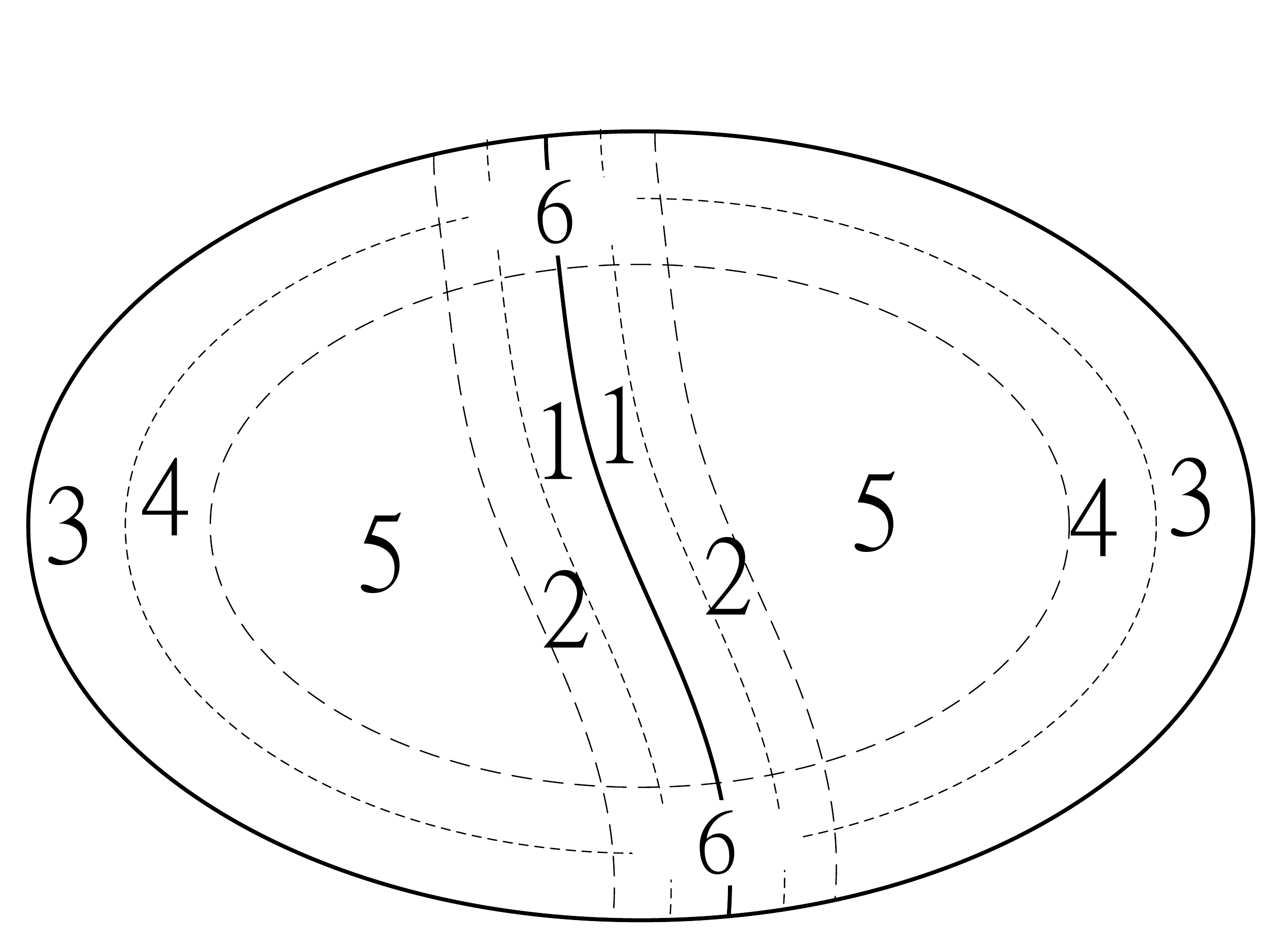}
\caption{The domain represents the domain $\Omega$ ad the region on the left side of the solid line $\partial V\cap\Omega$ represents the domain $V$. The distance between each type of line is $\epsilon^\gamma$. The $k$-region represents the sub-domain $\Omega_k$.}
\label{fig:domain}
\end{center}
\end{figure}

\noindent
3. We claim that the sequence $\{\mf{w}_\epsilon\}$ converges to $\mathbf v_0$ in $\left(L^2(\Omega)\right)^3$
 and
 \be
 {\mathcal{N}}[\mf{w}_\ep] = N+ n_\ep, \qquad {\mathcal{M}}[\mf{w}_\ep] = M+ m_\ep
 \label{eq:N-M-error}
 \ee
 with $n_\epsilon=O(\epsilon)$ and $m_\epsilon=O(\epsilon)$.
 
 We partition the domain $\Omega$ into sub-domains $\Omega_k$  (See  Figure \ref{fig:domain}):
 \begin{align*}
  \Omega_1 &= \{x\in\Omega: |d(x)| < \epsilon^\gamma, |d_b(x)| > 2\epsilon^\gamma\},\\
  \Omega_2 &= \{x\in\Omega:  \epsilon^\gamma \le |d(x)|\le 2\epsilon^\gamma, |d_b(x)| > 2\epsilon^\gamma\},\\
  \Omega_3 &= \{x\in\Omega: |d(x)|>2\epsilon^\gamma, |d_b(x)| < \epsilon^\gamma\},\\
  \Omega_4 &= \{x\in\Omega: |d(x)|>2\epsilon^\gamma, \epsilon^\gamma\le |d_b(x)| \le  2\epsilon^\gamma\},\\
   \Omega_5 &= \{x\in\Omega: |d(x)|>2\epsilon^\gamma, |d_b(x)| > 2\epsilon^\gamma\},\\
  \Omega_6 &= \{x\in\Omega: |d(x)|<2\epsilon^\gamma, |d_b(x)| < 2\epsilon^\gamma\}
 \end{align*}
 and estimate the following integral term by term on each subdomain $\Omega_k$:
$$
 \int_\Omega | \mf{w}_\epsilon - \mathbf v_0|\,dx
 =\sum_{k=1}^6 \int_{\Omega_k} | \mf{w}_\epsilon - \mathbf v_0|\,dx.
$$
We calculate
\begin{eqnarray*}
 \int_{\{x\in\Omega_1\cup\Omega_2:\,0<d(x)<2\epsilon^\gamma\}} | \mf{w}_\epsilon - \mathbf v_0|\,dx
&\le&  \int_{ \{ 0 < d(x) < 2 \epsilon^\gamma\}}\left|  \zeta(\frac{d(x)}{\epsilon^\gamma}) \left( \boldsymbol\eta(\frac{d(x)}{\epsilon})- \mathbf b\right) \right|\,dx \\
&\le&  \int_{ \{ 0 < d(x) < 2\epsilon^\gamma\}} \left|\boldsymbol\eta(\frac{d(x)}{\epsilon}) - \mathbf b\right|\,dx \\
&=& \int_0^{2\epsilon^\gamma} \left|\boldsymbol\eta(\frac{s}{\epsilon}) - \mathbf b\right|\,\mathcal H^2\{d(x)=s\}\,ds \\
&\le& \epsilon\,\left(  \max_{0\le s\le 2\epsilon^\gamma} \mathcal H^2\{d(x)=s\}\right)\, \int_0^{2 \epsilon^{\gamma-1}} \left|\boldsymbol\eta(t) - \mathbf b\right|\,dt
=O(\epsilon)
\end{eqnarray*}
because  the exponential decay estimates \eqref{eq:ex-estimate1}, \eqref{eq:ex-estimate2} and  \eqref{eq:measure}.
Similarly, we have
$$\int_{ \{ x\in\Omega_1\cup\Omega_2:\, -2\epsilon^\gamma < d(x) < 0\}} | \mf{w}_\epsilon - \mathbf v_0|\,dx
   \le \epsilon\,\left(  \max_{- 2\epsilon^\gamma\le s\le 0} \mathcal H^2\{d(x)=s\}\right)\, \int^0_{-2\epsilon^{\gamma-1}} \left|\boldsymbol\eta(t) - \mathbf a\right|\,dt
= O(\epsilon),
 $$
 $$\int_{ \{ x\in\Omega_3\cup\Omega_4:\, -2\epsilon^\gamma < d_b(x) < 0\}} | \mf{w}_\epsilon - \mathbf v_0|\,dx
   \le \epsilon\,\left(  \max_{- 2\epsilon^\gamma\le s\le 0} \mathcal H^2\{d_b(x)=s\}\right)\, \int_0^{2\epsilon^{\gamma-1}} \left|\boldsymbol\eta_{\mathbf a}(t) - \mathbf a\right|\,dt
= O(\epsilon),
 $$
 $$\int_{ \{ x\in\Omega_3\cup\Omega_4: 0 < d(x) < 2 \epsilon^\gamma\}} | \mf{w}_\epsilon - \mathbf v_0|\,dx
   \le \epsilon\,\left(  \max_{ 0 \le s\le 2\epsilon^\gamma} \mathcal H^2\{d(x)=s\}\right)\, \int_0^{2\epsilon^{\gamma-1}} \left|\boldsymbol\eta_{\mathbf b}(t) - \mathbf b\right|\,dt
= O(\epsilon).
 $$
We calculate
$$ \int_{\Omega_5} \left|  \mf{w}_\epsilon - \mathbf v_0\right|\,dx = 0$$
and
$$ \int_{\Omega_6} \left|  \mf{w}_\epsilon - \mathbf v_0\right|\,dx \le \max_{\Omega_6} ( |\mf{w}_\epsilon| + |\mathbf v_0|) |\Omega_6| = O(\epsilon^{2\gamma}).$$
Finally, we calculate
\begin{eqnarray*}
 &&\left| \int_\Omega |\mf{w}_\epsilon|^2 - |\mf{v}_0|^2\,dx\right|
 = \left| \int_\Omega ( |\mf{w}_\epsilon| + |\mf{v}_0|)( |\mf{w}_\epsilon| - |\mf{w}_0|)\,dx\right| \\
 &&\le \max_\epsilon (\|\mf{w}_\epsilon\| + \|\mf{v}_0\|)
    \int_\Omega \left| \mf{w}_\epsilon - \mf{v}_0\right| \,dx
    =O(\epsilon).
\end{eqnarray*}
Similarly, we have
$$
\int_\Omega ( |w_{\epsilon,1}|^2 - |w_{\epsilon,-1}|^2 )
- ( |w_{0,1}|^2 - |w_{0,-1}|^2 )   \,dx = O(\epsilon).
$$

\noindent
4. We rewrite $\ml G_\epsilon[\mf{w}_\epsilon]$ and estimate them term by term:
$$
\ml G_\epsilon[\mf{w}_\epsilon] = \int_\Omega \epsilon |\nabla \mf{w}_\epsilon|^2 + \frac{1}{\epsilon}W(\mf{w}_\epsilon)\,dx
  = \sum_{k=1}^6 \int_{\Omega_k} \epsilon |\nabla \mf{w}_\epsilon|^2 + \frac{1}{\epsilon}W(\mf{w}_\epsilon)\,dx.
$$
 Using the coarea formula, \eqref{eq:measure} and Lemma \ref{lemma:4.3}, we calculate  the first term
\begin{eqnarray*}
 && \lim_{\epsilon\to 0}  \int_{\Omega_1} \epsilon |\nabla \mf{w}_\epsilon|^2 + \frac{1}{\epsilon}W(\mf{w}_\epsilon)\,dx \\
 && = \lim_{\epsilon\to 0} \frac{1}{\epsilon} \int_{\Omega_1}  \left|\boldsymbol\eta'\left(\frac{d(x)}{\epsilon}\right)\right|^2 + W\left(\boldsymbol\eta\left(\frac{d(x)}{\epsilon}\right)\right)\,dx \\
 &&= \lim_{\epsilon\to 0} \frac{1}{\epsilon} \int_{-\epsilon^\gamma}^{\epsilon^\gamma}  \left[  \left|\boldsymbol\eta'\left(\frac{s}{\epsilon}\right)\right|^2 + W\left(\boldsymbol\eta\left(\frac{s}{\epsilon}\right)\right) \right] \,\mathcal H^2\{x\in\Omega_1:\,d(x)=s\} \,ds \\
 &&=\lim_{\epsilon\to 0} \int_{-\epsilon^{\gamma-1}}^{\epsilon^{\gamma-1}} \left( |\boldsymbol\eta'(t)|^2 + W(\boldsymbol\eta(t))\right)  \,\mathcal H^2\{x\in\Omega_1:\,d(x)=\epsilon t\} \,dt \\
 &&= 2 g(\mathbf a,\mathbf b)\,\lim_{\epsilon\to 0}\max_{|s|<\epsilon^\gamma} \mathcal H^2\{x\in\Omega:\,d(x)= s\}\\
 &&\le 2 g(\mathbf a,\mathbf b)\,\textrm{Per}_\Omega(\mf{w}=\mathbf a).
\end{eqnarray*}
Similarly, for boundary layer, we also have
\begin{eqnarray*}
 && \lim_{\epsilon\to 0}  \int_{\Omega_3} \epsilon |\nabla \mf{w}_\epsilon|^2 + \frac{1}{\epsilon}W(\mf{w}_\epsilon)\,dx \\
 && = \lim_{\epsilon\to 0} \frac{1}{\epsilon} \int_{\Omega_3}  \left|\tilde{\boldsymbol\eta}'\left(\frac{d_b(x)}{\epsilon}\right)\right|^2 + W\left(\tilde{\boldsymbol\eta}\left(\frac{d_b(x)}{\epsilon}\right)\right)\,dx \\
 &&= \lim_{\epsilon\to 0} \frac{1}{\epsilon} \int_{-\epsilon^\gamma}^{\epsilon^\gamma}
  \left[  \left| \tilde{\boldsymbol\eta}'\left(\frac{s}{\epsilon}\right)\right|^2 + W\left(\tilde{\boldsymbol\eta}\left(\frac{s}{\epsilon}\right)\right) \right] \,\mathcal H^2\{x\in\Omega_3:\,d_b(x)=s\} \,ds \\
 &&=\lim_{\epsilon\to 0} \int_{-\epsilon^{\gamma-1}}^{\epsilon^{\gamma-1}} \left( |\tilde{\boldsymbol\eta}'(t)|^2 + W(\tilde{\boldsymbol\eta}(t))\right)  \,\mathcal H^2\{x\in\Omega_3:\,d_b(x)=\epsilon t\} \,dt \\
 &&=\lim_{\epsilon\to 0}  \left(\int_{0}^{\epsilon^{\gamma-1}} \left( |\boldsymbol\eta_{\mathbf a}'(t)|^2 + W(\boldsymbol\eta_{\mathbf a}(t))\right)  \,\mathcal H^2\{x\in\Omega_3:\,d_b(x)=-\epsilon t\} \,dt \right. \\
 && \qquad\qquad + \left. \int_{0}^{\epsilon^{\gamma-1}} \left( |\boldsymbol\eta_{\mathbf b}'(t)|^2 + W(\boldsymbol\eta_{\mathbf b}(t))\right)  \,\mathcal H^2\{x\in\Omega_3:\,d_b(x)=\epsilon t\} \,dt   \right) \\
 &&\le 2 g(\mathbf 0,\mathbf a)\,\lim_{\epsilon\to 0}\max_{0<s<\epsilon^\gamma} \mathcal H^2\{x\in\Omega:\,d_b(x)= -s\} \\
 &&\qquad\qquad + 2g(\mathbf 0,\mathbf b)  \,\lim_{\epsilon\to 0}\max_{0<s<\epsilon^\gamma} \mathcal H^2\{x\in\Omega:\,d_b(x)= s\}\\
 &&= 2 g(\mathbf 0,\mathbf a)\,\mathcal H^2(\{x\in\partial\Omega:\,\mathbf v_0(x)=\mathbf a\})
    + 2 g(\mathbf 0,\mathbf b)\,\mathcal H^2(\{x\in\partial\Omega:\,\mathbf v_0(x)=\mathbf b\}).
\end{eqnarray*}

Applying the Taylor expansion of the function $W$ around $\mathbf a$ and $\mathbf b$ and using the exponential decay estimates \eqref{eq:ex-estimate1} and \eqref{eq:ex-estimate2}
\begin{align*}
  & \lim_{\epsilon\to0}  \int_{\Omega_2} \epsilon |\nabla \mf{w}_\epsilon|^2 + \frac{1}{\epsilon}W(\mf{w}_\epsilon)\,dx
 =0, \\
 & \lim_{\epsilon\to0}  \int_{\Omega_4} \epsilon |\nabla \mf{w}_\epsilon|^2 + \frac{1}{\epsilon}W(\mf{w}_\epsilon)\,dx
 =0.
\end{align*}
Since $\mf{w}_\epsilon$ equals to $\mathbf a$ or $\mathbf b$  on $\Omega_5$, we have
$$ \int_{\Omega_5} \epsilon |\nabla \mf{w}_\epsilon|^2 + \frac{1}{\epsilon}W(\mf{w}_\epsilon)\,dx
 =O(\ep).$$

$$
\int_{\Omega_6} \epsilon |\nabla \mf{w}_\epsilon|^2 + \frac{1}{\epsilon}W(\mf{w}_\epsilon)\,dx
 = O(\epsilon)+ O(\epsilon^{2\gamma-1})
$$
Thus, we obtain
$$ \limsup_{\epsilon\to 0} \ml G_\epsilon[{\bf w}_\epsilon] \le  \ml G_0[\mathbf v_0].$$
Combining the result of the lower semi-continutiy, we have
$$  \lim_{\epsilon\to 0} \ml G_\epsilon[{\bf w}_\epsilon] = \ml G_0[\mathbf v_0].$$

\noindent
5. Finally, we  modify the layer function ${\bf w}_\epsilon$ by adding some smooth function with compact support in order to fit the conservation constraints. We choose two smooth functions $\boldsymbol\varphi:\Omega\to \mathbb R^3$ and $\boldsymbol\psi:\Omega\to \mathbb R^3$  such that $\boldsymbol\varphi,\boldsymbol\psi$ satisfy  the following conditions:
\begin{enumerate}
\item[(i)] The function $\boldsymbol\varphi$ has compact support in $V$ and the function $\boldsymbol\psi$ has compact support in $\Omega\backslash V$.
\item[(ii)] There exists $\delta>0$ such that all the components of the function $(\mathbf a+ \gamma \boldsymbol \varphi(x))$ and the function   $(\mathbf b+ \gamma \boldsymbol \psi(x))$ are all nonnegative  on $\Omega$ whenever $|\gamma|<\delta$.
\item[(iii)] The matrix
          $$
           \left(\begin{array}{cc}
            \int_V ( a_1 \varphi_1  +  a_0 \varphi_{0} + a_{-1}\varphi_{-1})\,dx
            & \int_{\Omega\backslash V} (b_1 \psi_1  +  b_0 \psi_{0} + b_{-1}\psi_{-1})\,dx \\
            \int_V (a_1 \varphi_1  - a_{-1}\varphi_{-1})\,dx
            & \int_{\Omega\backslash V} (b_1\psi_1  -b_{-1}\psi_{-1})\,dx
           \end{array}\right)
          $$
          is invertible where $\mathbf a=(a_1,a_0,a_{-1})$ and $\mathbf b=(b_1,b_0,b_{-1})$.
\end{enumerate}

For each $\epsilon$ small enough, we would like to find $\alpha_1,\alpha_2$ such that the function $(\mf{w}_\epsilon + \alpha_1 \boldsymbol\varphi + \alpha_2 \boldsymbol\psi)$ satisfies the constraints of total mass and total magnetization. That is
\begin{equation*}
\left\{\begin{array}{rcl}
 \int_\Omega |w_{\epsilon,1} + \alpha_1 \varphi_1 + \alpha_2 \psi_1|^2 + |w_{\epsilon,0} + \alpha_1 \varphi_0 + \alpha_2 \psi_0|^2
    + |w_{\epsilon,-1} + \alpha_1 \varphi_{-1} + \alpha_2 \psi_{-1}|^2\,dx &=& N , \\
 \int_\Omega    |w_{\epsilon,1} + \alpha_1 \varphi_1 + \alpha_2 \psi_1|^2 - |w_{\epsilon,-1} + \alpha_1 \varphi_{-1} + \alpha_2 \psi_{-1}|^2\,dx &=& M.
\end{array}\right.
\end{equation*}
Because of \eqref{eq:N-M-error}, we obtain the system
\begin{eqnarray*}
&& f_1(\alpha_1,\alpha_2)\equiv\alpha_1^2  \int_\Omega (\varphi_1^2+\varphi_0^2+\varphi_{-1}^2)\,dx
+ 2\alpha_1  \int_\Omega ( w_{\epsilon,1}\varphi_1  +  w_{\epsilon,0}\varphi_{0} + w_{\epsilon,1}\varphi_{-1})\,dx\\
&&\qquad\qquad +\alpha_2^2  \int_\Omega (\psi_1^2+\psi_0^2+\psi_{-1}^2)\,dx
+ 2\alpha_2 \int_\Omega (w_{\epsilon,1}\psi_1  +  w_{\epsilon,0}\psi_{0} + w_{\epsilon,-1}\psi_{-1})\,dx
= -n_\epsilon, \\
&&f_2(\alpha_1,\alpha_2)\equiv \alpha_1^2  \int_\Omega (\varphi_1^2 - \varphi_{-1}^2)\,dx
+ 2\alpha_1 \int_\Omega (w_{\epsilon,1}\varphi_1  - w_{\epsilon,-1}\varphi_{-1})\,dx\\
&&\qquad \qquad +\alpha_2^2  \int_\Omega (\psi_1^2 - \psi_{-1}^2)\,dx
+ 2\alpha_2 \int_\Omega (w_{\epsilon,1}\psi_1  - w_{\epsilon,-1}\psi_{-1})\,dx
= -m_\epsilon.
\end{eqnarray*}
It is observed that $(\alpha_1,\alpha_2)=(0,0)$ is a solution of the system
$$
 \left\{\begin{array}{rcl}
  f_1(\alpha_1,\alpha_2) &=& 0,\\
  f_2(\alpha_1,\alpha_2)&=& 0,
 \end{array}\right.
$$
and the Jacobian matrix
$$
\left(\begin{array}{cc}
  \frac{\partial f_1}{\partial \alpha_1}(0,0) & \frac{\partial f_1}{\partial \alpha_2}(0,0)\\
  \frac{\partial f_2}{\partial \alpha_1}(0,0) &  \frac{\partial f_2}{\partial \alpha_2}(0,0)
 \end{array}\right)
 = 2
           \left(\begin{array}{cc}
            \int_V ( a_1 \varphi_1  +  a_0 \varphi_{0} + a_{-1}\varphi_{-1})\,dx
            & \int_{\Omega\backslash V} (b_1 \psi_1  +  b_0 \psi_{0} + b_{-1}\psi_{-1})\,dx \\
            \int_V (a_1 \varphi_1  - a_{-1}\varphi_{-1})\,dx
            & \int_{\Omega\backslash V} (b_1\psi_1  -b_{-1}\psi_{-1})\,dx
           \end{array}\right)
$$
 is invertible. According to the Inverse function theorem, for $\epsilon$ small enough the system
 $$
 \left\{\begin{array}{rcl}
  f_1(\alpha_1,\alpha_2) &=& -n_\epsilon,\\
  f_2(\alpha_1,\alpha_2)&=& -m_\epsilon,
 \end{array}\right.
$$
 is solvable. Furthermore, for each $\epsilon$ small, the corresponding $\alpha_1(\epsilon),\alpha_2(\epsilon)$ are of order $O(\epsilon^{2\gamma})$. Define the recovering sequence $\{\mathbf v_\epsilon\}$ by
 $$ \mathbf v_\epsilon(x) = \mf{w}_\epsilon(x) + \alpha_1(\epsilon)\, \boldsymbol\varphi(x) + \alpha_2(\epsilon) \,\boldsymbol\psi(x).$$
By the choice of the function $\mathbf v_\epsilon$, it will satisfy the constraints of total mass and total magnetization. Finally, we calculate

\begin{eqnarray*}
 \lim_{\epsilon\to 0}\ml G_\epsilon[\mathbf v_\epsilon] &=&  \lim_{\epsilon\to 0}\int_\Omega \epsilon \left|\nabla (\mf{w}_\epsilon + \alpha_1(\epsilon)\, \boldsymbol\varphi + \alpha_2(\epsilon) \,\boldsymbol\psi ) \right|^2 + \frac{1}{\epsilon} W(\mf{w}_\epsilon + \alpha_1(\epsilon)\, \boldsymbol\varphi + \alpha_2(\epsilon) \,\boldsymbol\psi)\,dx \\
  &=&   \lim_{\epsilon\to 0} \int_{\Omega_5} \epsilon \left| \nabla(\mf{w}_\epsilon + \alpha_1(\epsilon)\, \boldsymbol\varphi + \alpha_2(\epsilon) \,\boldsymbol\psi )\right|^2 + \frac{1}{\epsilon} W(\mf{w}_\epsilon + \alpha_1(\epsilon)\, \boldsymbol\varphi + \alpha_2(\epsilon) \,\boldsymbol\psi)\,dx \\
  &&+ \lim_{\epsilon\to 0}\int_{\Omega\backslash\Omega_5} \epsilon \left|\nabla(\mf{w}_\epsilon + \alpha_1(\epsilon)\, \boldsymbol\varphi + \alpha_2(\epsilon) \,\boldsymbol\psi )\right|^2 + \frac{1}{\epsilon} W(\mf{w}_\epsilon + \alpha_1(\epsilon)\, \boldsymbol\varphi + \alpha_2(\epsilon) \,\boldsymbol\psi)\,dx
\end{eqnarray*}
 When $\epsilon$ is small enough, the support of the function $\boldsymbol\varphi$ and $\boldsymbol\psi$ is contained in $\Omega_5$.
We find that
\begin{eqnarray*}
 && \lim_{\epsilon\to 0}\int_{\Omega_5} \epsilon \left| \nabla(\mf{w}_\epsilon + \alpha_1(\epsilon)\, \boldsymbol\varphi + \alpha_2(\epsilon) \,\boldsymbol\psi )\right|^2 + \frac{1}{\epsilon} W(\mf{w}_\epsilon + \alpha_1(\epsilon)\, \boldsymbol\varphi + \alpha_2(\epsilon) \,\boldsymbol\psi)\,dx \\
 &=& \lim_{\epsilon\to 0}\int_{\Omega_5} \frac{1}{\epsilon} W(\mf{w}_\epsilon + \alpha_1(\epsilon)\, \boldsymbol\varphi + \alpha_2(\epsilon) \,\boldsymbol\psi)\,dx \\
 &\le& C  \lim_{\epsilon\to 0}\int_{\Omega_5}  \frac{1}{\epsilon}  | \alpha_1(\epsilon)\, \boldsymbol\varphi + \alpha_2(\epsilon) \,\boldsymbol\psi|^2\,dx\\
 &=& O(\epsilon).
\end{eqnarray*}
Since $ \alpha_1(\epsilon)\, \boldsymbol\varphi + \alpha_2(\epsilon) \,\boldsymbol\psi =0$ on $\Omega\backslash\Omega_5$ for $\epsilon$ small enough, we obtain
$$ \lim_{\epsilon\to 0}\ml G_\epsilon[\mathbf v_\epsilon] = \ml G_0[\mathbf v_0].$$
Thus,  $\{\mf{v}_\ep\}$  is a recovery sequence for the functional $\ml G_0$ about $\mf{v}_0\in \ms{A}_0\cap BV(\Omega)$. 
  
\noindent
6. Finally, we need to construct a recovery sequence $\{\mf{v}_\ep\}\subset\ms{A}$ for the case when $\mf{v}_0\in \ms{A}_0\backslash BV(\Omega)$. This means that the set $V\!:=\{\mf v_0=\mf a\}$ has an infinite perimeter.  Our approach is to construct a sequence $\{\tilde{\mf{v}}_\ep\}$ of the form 
$$ \tilde{\mathbf v}_\epsilon(x) = \mf{w}_\epsilon(x) + \alpha_1(\epsilon)\, \boldsymbol\varphi(x) + \alpha_2(\epsilon) \,\boldsymbol\psi(x),$$
then to choose $\mf v_\ep = (|\tilde{v}^\ep_1|,|\tilde{v}^\ep_0|,|\tilde{v}^\ep_{-1}|)$.
Here, the sequence $\{\mf{w}_\ep\}\subset (C^\infty(\Omega))^3$ is taken to be $\mf w_\ep := \mf v_0*\phi_\ep$, where $\phi_\ep$ is a standard modifier; and $\boldsymbol\varphi:\Omega\to \mathbb R^3$ and $\boldsymbol\psi:\Omega\to \mathbb R^3$ are two smooth functions chosen in a way that the matrix
$$
           \left(\begin{array}{cc}
            \int_\Omega ( v^0_1 \varphi_1  +  v^0_0 \varphi_{0} + v^0_{-1}\varphi_{-1})\,dx
            & \int_\Omega (v^0_1 \psi_1  +  v_0 \psi_{0} + v^0_{-1}\psi_{-1})\,dx \\
            \int_\Omega (v^0_1 \varphi_1  - v^0_{-1}\varphi_{-1})\,dx
            & \int_\Omega (v^0_1\psi_1  -v^0_{-1}\psi_{-1})\,dx
           \end{array}\right)
          $$
          is invertible. The coefficients $\alpha_i(\ep)$ are chosen so that $\tilde{\mathbf v}_\epsilon(x)$ satisfies the conservation constraints.
It is clear that  $\mf{w}_\ep\to \mf{v}_0$ in $(L^2(\Omega))^3$. 
We thus get $$\ml N[\mf w_\ep] \to \ml N[\mf v_0] = N, \quad \ml M[\mf w_\ep] \to \ml M[\mf v_0] = M.$$   
By the implicit function theorem, we obtain that there exists some $\ep_0>0$ and  two corresponding functions $\alpha_1(\ep), \alpha_2(\ep)$ for $0\le \ep<\ep_0$ with $\alpha_i(\ep)\to0$ as $\ep \to 0$ such that  the corresponding function $\tilde{\mf{v}}_\ep$ satisfy
$$
 {\mathcal{N}}[\tilde{\mf{v}}_\ep] = N, \qquad {\mathcal{M}}[\tilde{\mf{v}}_\ep] = M.
$$
From $\alpha_i(\ep)\to 0$ as $\ep \to 0$ and $\mf w_\ep \to \mf v_0$, we get $\tilde{\mf v}_\ep \to \mf v_0$ in $(L^2(\Omega))^3$.  Now set $\mf{v}_\ep = (|\tilde{v}^\ep_1|,|\tilde{v}^\ep_0|,|\tilde{v}^\ep_{-1}|)$.  
It is obvious that the sequence $\{\mf{v}_\ep\}$ also converges to $\mf{v}_0$ in $(L^2(\Omega))^3$ and satisfies the conservation constraints
$$
 {\mathcal{N}}[\mf{v}_\ep] = N, \qquad {\mathcal{M}}[\mf{v}_\ep] = M.
$$
By the lower-semicontinuity of $\ml G_\ep$ that has been proved in the first part of this theorem, we have 
$$ +\infty=2 g(\mf{a},\mf{b})\textrm{Per}(\mf{v}_0=\mf{a})\le \ml{G}_0(\mf{v}_0)\le \liminf_{\ep\to 0} \ml{G}_\ep(\mf{v}_\ep).$$
This implies that
$$ \lim_{\ep\to 0} \ml{G}_\ep(\mf{v}_\ep)=+\infty=\ml{G}_0(\mf{v}_0).$$
This completes the proof of the $\Gamma$-convergence of the sequence $\{\ml G_\epsilon\}$ to $\ml G_0$. 
\begin{flushright}
{$\Box$}
\end{flushright}

\paragraph{Proof of Theorem \ref{Thm:compactness}}
  Set $R_1:=\max(R,|\mathbf a|+1,|\mathbf b|+1)$ where $R$ is given in \eqref{cond:W3} and define a truncating function by
 $$ \mathbf v_\epsilon:=\mathbf u_\epsilon\,\chi_{|\mathbf u_\epsilon|\le R_1}.$$
 By using \eqref{cond:W3} and the assumption \eqref{eq:boundedness}, we find
 \begin{align}
  \int_\Omega |\mathbf v_\epsilon - \mathbf u_\epsilon|^2 &= \int_{\{|\mathbf u_\epsilon|>R_1\}} |\mathbf u_\epsilon|^2\,dx \nonumber \\
  &\le \frac{1}{C} \int_{\{|\mathbf u_\epsilon|>R_1\}} W(\mathbf u_\epsilon) \,dx \nonumber \\
  &\le \frac{\epsilon}{C}\ml G_\epsilon[\mathbf u_\epsilon]= \frac{\epsilon C_0}{C}\to 0 \label{eq:v-u}
 \end{align}
 as $\epsilon\to 0$. Therefore, we only need to show the precompactness of the $L^\infty$-sequence $\{\mathbf v_\epsilon\}$ in $(L^2(\Omega))^3$.  In order to achieve this, we will apply the compactness result for Young measures: the $L^\infty$-boundedness of the sequence $\{\mathbf v_\epsilon\}$ implies there exist a subsequence $\{\mathbf v_{\epsilon_j}\}$ and a Young measure $\mu$ such that
 \be
  \textrm{ if $f$ is a continuous function on $\mathbb R^3$, then $f(\mathbf v_\epsilon)\to\left(x\to \int_\Omega f(y)\,d\mu_x(y)\right)$ in $L^\infty$ weakly*. }\label{eq:young-measure}
 \ee
 In the light of \eqref{eq:v-u}, we conclude that the measure of the set $\{x\in\Omega:\,|\mathbf u_\epsilon(x)| > R_1 \}$ tends to zero. We find
 \begin{eqnarray*}
 \lim_{\epsilon\to 0}\int_\Omega W(\mathbf v_\epsilon)\,dx &=& \lim_{\epsilon\to 0} \left( \int_{\{|\mathbf u_\epsilon|\le R_1\}} W(\mathbf u_\epsilon)\,dx + \int_{\{|\mathbf u_\epsilon| > R_1 \}}  W(\mathbf 0)\,dx  \right) \\
 &\le& \lim_{\epsilon\to 0} \int_\Omega W(\mathbf u_\epsilon)\,dx \\
 &\le& \epsilon\, \ml G_\epsilon[\mathbf u_\epsilon]\to 0.
 \end{eqnarray*}
Because of \eqref{cond:W1}, we know that
 \be \mu_x =\theta(x)\,\delta_{y=\mathbf a} + (1-\theta(x))\,\delta_{y=\mathbf b}\quad\textrm{a.e. }x\in \Omega
\label{eq:4.21}
 \ee
 where $0\le \theta(x)\le 1$. 
 
On the other hand, from  Lemma \ref{lemma:4.2} we know that $\varphi_{\mathbf a}$ is Lipchitz continuous and leads to the $L^\infty$-boundedness of the sequence $\{\varphi_{\mathbf a}\circ\mathbf v_\epsilon\}$. Thus, there exist a subsequence $\{\varphi_{\mathbf a}\circ\mathbf v_{\epsilon_k}\}$ and a Young measure $\nu_x$ such that the subsequence $\{\varphi_{\mathbf a}\circ\mathbf v_{\epsilon_k}\}$ converges to a Young measure $\nu_x$. Because of \eqref{eq:4.21}, we can express the Young measure $\nu_x$ as
 $$ \nu_x =\theta(x)\,\delta_{y=\varphi_{\mathbf a}(\mathbf a)} + (1-\theta(x))\,\delta_{y=\varphi_{\mathbf a}(\mathbf b)}\quad\textrm{a.e. }x\in \Omega.$$
 Now, we are going to show the sequence $\{\varphi_{\mathbf a}\circ\mathbf u_\epsilon\}$ is bounded in $BV(\Omega)$. First, we estimate 
 \begin{eqnarray*}
 \int_\Omega \left| \varphi_{\mathbf a}(\mathbf u_\epsilon(x)) \right|\,dx 
 &\le& \int_\Omega  \left|\varphi_{\mathbf a}(\mathbf v_\epsilon(x))\right|\,dx +  \int_\Omega \left| \varphi_{\mathbf a}(\mathbf u_\epsilon(x)) - \varphi_{\mathbf a}(\mathbf v_\epsilon(x))\right|\,dx  \\
&\le&  |\Omega| \|\varphi_{\mathbf a}\circ\mathbf v_\epsilon\|_{L^\infty} + \int_\Omega |\varphi_{\mathbf a}'| \left| \mathbf u_\epsilon-\mathbf v_\epsilon\right|\,dx 
 \end{eqnarray*}
 and find $\varphi_{\mathbf a}\circ\mathbf u_\epsilon$ is bounded in $L^1(\Omega)$ since the function $\varphi_{\mathbf a}(x)$ is Lipchitz continuous on $\mathbb R^3$ and $\| \mathbf u_\epsilon-\mathbf v_\epsilon\|_{L^2}\to 0$ as $\epsilon\to 0$.
 Next, we estimate
 \begin{eqnarray*}
  \int_\Omega |\nabla (\varphi_{\mathbf a}\circ\mathbf u_\epsilon)(x))|\,dx &\le& \int_\Omega \sqrt{W(\mathbf u_\epsilon(x))}\,|\nabla \mathbf u_\epsilon| \,dx \\
   &\le&\frac{1}{2} \int_\Omega \epsilon|\nabla\mathbf u_\epsilon|^2 + \frac{1}{\epsilon} W(\mathbf u_\epsilon)\,dx \\
   &\le& \frac{1}{2} \ml G_\epsilon[\mathbf u_\epsilon]\le \frac{C_0}{2}<+\infty.
 \end{eqnarray*}
Therefore, the sequence $\{\varphi_{\mathbf a}\circ\mathbf u_\epsilon\}$ is bounded in $BV$-norm.  There exists a subsequence (without abusing the notation, we still denote it by the same sequence) and a function $h\in L^1(\Omega)$ such that
$$ \varphi_{\mathbf a}\circ\mathbf u_\epsilon\to h \qquad\textrm{in }L^1(\Omega).$$
Since the function $\varphi_{\mathbf a}$ is Lipchitz continuous on $\mathbb R^3$ and $\| \mathbf u_\epsilon-\mathbf v_\epsilon\|_{L^2}\to 0$ as $\epsilon\to 0$, we have
$$ \varphi_{\mathbf a}\circ\mathbf v_\epsilon\to h \qquad\textrm{in }L^1(\Omega).$$
In consequence, 
the Young measure $\nu_x$ associated with $\varphi_{\mathbf a}\circ\mathbf v_\epsilon$ is just a point mass a.e. $x\in \Omega$, that is
$$ \nu_x=\delta_{h(x)}\quad\textrm{a.e. }x\in \Omega$$
 and the function $\theta(x)=\chi_U$ for some measurable set $U\subset\Omega$.
 Thus, the corresponding Young measure $\mu_x$ could be represented by
 the function
$$ \mathbf u_0:=\mathbf a\,\chi_U  + \mathbf b\,\chi_{\Omega\backslash U}.$$
The definition of the convergence in Young measures to a function $\mathbf u_0$ give us that
$$ \mathbf v_\epsilon\to \mathbf u_0 \qquad\textrm{weakly in }L^p(\Omega) \textrm{ for }1\le p<\infty$$
and
$$ \|\mathbf v_\epsilon\|_{L^p} \to \|\mathbf u_0\|_{L^p}\qquad\textrm{ for }1\le p<\infty.$$
It follows that $\mathbf v_\epsilon$ converges to $\mathbf u_0$ strongly in $(L^p(\Omega))^3$ for $1\le p<\infty$. From \eqref{eq:v-u}, we have
$\mathbf u_\epsilon$ converges to $\mathbf u_0$ in $(L^2(\Omega))^3$. Because $\mf{u}_0$ also satisfies the conservation constraints $ {\mathcal{N}}[\mf{u}_0] = N$ and ${\mathcal{M}}[\mf{u}_0] = M.
$, we can conclude $\mf{u}_0\in\ms{A}_0$.
 This completes the proof.
\begin{flushright}
{$\Box$}
\end{flushright}

\paragraph{Acknowledge} The authors thank the National Center for Theoretical Sciences of the Republic of China for helpful supports. The authors would like to thanks Prof.~Peter Sternberg for his helpful suggestions.  I.C.~and C.C.~are partially supported by the Ministry of Science and Technology of the Republic of China (NSC102-2115-M-002-016-MY2, NSC 102-2811-M-009-030)   
\bibliographystyle{plain}
\bibliography{../bec}

\begin{thebibliography}{99}

\bibitem{AR}  Aftalion, A., Royo-Letelier, J., A minimal interface problem arising from a two component Bose-Einstein condensate via $\Gamma$- convergence, {\it Calc. Var. Partial Differential Equation} {\bf 52}(1), 165-197 (2015).

\bibitem{Anderson95} Anderson, M.H.,  Ensher, J.R., Mathews, M.R., Wieman, C.E., Cornell, E.A., Observation of Bose-Einstein condensation in a dilute atomic vapor, {\it Science } {\bf 269}, 198-201 (1995).



\bibitem{Baldo} Baldo, S., Minimal interface criterion for phase transitions in mixtures of Cahn-Hilliard fluids, {\it Ann. Inst. H. Poincar\'e Anal. Non Lin\'eaire} {\bf 7}(2), 67-90 (1990).

\bibitem{BaoCai2012}
Bao, W., Cai, Y., Mathematical theory and numerical methods for Bose-Einstein condensation, {\it Kinet. Relat. Models} {\bf 6}(1), 1-135 (2012).

\bibitem{Bao2013} Bao, W., Chern, I-L., Zhang, Y., Efficient numerical
  methods for computing ground states of spin-1 Bose-Einstein
  condensates based on their characterizations, {\it J. 
  Comput. Phys.} {\bf 253}, 189-208 (2013).

\bibitem{BaoCai2008} Bao, W., Lim, F.Y., Computing ground states of
  spin-1 {B}ose-{E}instein condensates by the normalized gradient
  flow, {\it SIAM J. Sci. Comput.} {\bf 30}(4), 1925-1948 (2008).

\bibitem{Bao2010} Bao, W., Zhang, Y., Dynamical laws of the coupled
  Gross-Pitaevskii equations for spin-1 {B}ose-{E}instein condensates,
  {\it Methods Appl. Anal.} {\bf 17}(1), 49--80 (2010).

\bibitem{bloch2008} Bloch, I., Dalibard, J., Zwerger, W., Many-body
  physics with ultracold gases, {\it Rev. Mod. Phys.} {\bf 80}(3), 885-964 (2008).

\bibitem{Bose1924}
Bose, S.N.,  Plancks gesetz und lichtquantenhypothese, {\it Z. Phys.} {\bf 26}, 178-181 (1924).

\bibitem{B} Braides, A.: $\Gamma$-Convergence for Beginners, {\it Oxford Lecture Series in Mathematics and its Applications}, Oxford 2002.

\bibitem{Cao2011} Cao, D., Chern, I-L., Wei, J.-C., On ground state
  of spinor {B}ose-{E}instein condensates, {\it NoDEA Nonlinear Differential Equations Appl.} {\bf 18}(4), 427-445 (2011).

\bibitem{chang2005} Chang, M.-S.,  Qin, Q., Zhang, W., You, L.,
  Chapman, M.S., Coherent spinor dynamics in a spin-1 {B}ose
  condensate, {\it Nat. Phys.} {\bf 1}, 111-116  (2005).

\bibitem{ChenChernWang2011}  Chen, J.-H., Chern,  I-L., Wang, W., Exploring ground
  states and excited states of spin-1 {B}ose-{E}instein condensates by
  continuation methods, {\it J. Comput. Phys.} {\bf 230}(6), 
   2222-2236 (2011).

\bibitem{CCW2014}
Chen, J.-H., Chern, I-L., Wang, W., A complete study of the ground state phase diagrams of spin-$1$ Bose-Einstein condensates in a magnetic field via
      continuation methods,  {\it J. Sci. Comput.} {\bf 64}, 35-54 (2014).

\bibitem{CS} Choksi, R., Sternberg, P., On the first and second variations of a nonlocal isoperimetric problem, {\it J. Reine Angew. Math.} {\bf 611}, 75-108 (2007).

%
\bibitem{DM}  Dal Maso, G.: An introduction to $\Gamma$-convergence, Birkh\"auser, Basel 1993.

\bibitem{Dalfovo99}
Dalfovo, F., Giorgini, S., Pitaevskii, L.P., Stringari, Theory of Bose-Einstein condensation in trapped gases, {\it Rev. Mod. Phys. }  {\bf 71}, 463-512 (1999).

\bibitem{Davis95}
 Davis, K. B., Mewes, M.-O.,  Andrews, M.R.,  van Druten, N. J.,  Durfee, D.D.,  Kurn, D. M.,  Ketterle, W., Bose-Einstein Condensation in a Gas of Sodium Atoms, {\it Phys. Rev. Lett. }  {\bf 75}, 3969-3973 (1995).
 
\bibitem{D}  De Giorgi, E., Convergence problems for functionals and operators. In: {\it Proceedings of the International Meeting on Recent Methods in Nonlinear Analysis} (Rome, 1978), pp.131-188, Pitagoria, Bologna (1979).

\bibitem{Einstein1924}
Einstein, A., Quantentheorie des einatomigen idealen gases, {\it Sitzungsberichte der Preussischen Akademie der Wissenschaften} {\bf 22}, 261-267 (1924).

\bibitem{Einstein1925}
Einstein, A.,  Quantentheorie des einatomigen idealen Gases, zweite abhandlung, {\it Sitzungsberichte der Preussischen Akademie der Wissenschaften} {\bf 1},  3-14 (1925).

\bibitem{Yau2010}
Erd\"os, L, Schlein, B., Yau, H.-T., Derivation of the Gross-Pitaevskii equation for the dynamics of Bose-Einstein condensate, {\it  Ann. of  Math. (2)} {\bf 172},  291-370 (2010).


\bibitem{E}  Evans, L.C.: Weak convergence methods for nonlinear partial differential equations, {\it CBMS Regional Conference Series in Mathematics}, Vol.71, 1990.

\bibitem{FT}  Fonseca, I.,  Tartar, L., The gradient theory of phase transitions for systems with two potential wells, {\it Proc. Roy. Soc. Edinburgh Sect. A} {\bf 111},  89-102 (1989).

\bibitem{Gross61}
 Gross, E.P., Structure of a quantized vortex in boson systems, {\it Nuovo Cimento (10)} {\bf 20}(3), 454-477 (1961).

\bibitem{G1}  Gurtin, M., On a theory of phase transitions with interfacial energy, {\it Arch. Rational Mech. Anal.}, {\bf 87}, 187-212 (1984).

\bibitem{G2} Gurtin, M.: Some results and conjectures in the gradient theory of phase transitions, {\it Metastability and Incompletely Posed Problem},  IMA Volumes in Mathematics and Its Applications, pp. 135-146. Springer, New York, 1987.

\bibitem{ho1998}  Ho, T.-L., Spinor {B}ose condensates in optical traps,
  {\it Phys. Rev. Lett.} {\bf 81}(4), 742-745  (1998).


\bibitem{I1} Ishige, K., Singular perturbations of variational problems of vector valued functions, {\it Nonlinear Anal.}, {\bf 23}, 1453-1466 (1994).

\bibitem{I2} Ishige, K., The gradient theory of the phase transitions in Cahn-Hilliard fluids with the Dirichlet boundary conditions, {\it SIAM J. Math. Anal.}, {\bf 27}(3), 620-637 (1996).

\bibitem{isoshima1999} Isoshima, T., Machida, K., Ohmi, T., Spin-domain
  formation in spinor {B}ose-{E}instein condensation, {\it Phys. Rev.
  A} {\bf 60} (6),  4857-4863 (1999).

\bibitem{Jacob2012} Jacob, D., Shao,  L., Corre, V., Zibold, T.,
  De Sarlo, L., Mimoun, E., Dalibard, J., Gerbier, F., Phase diagram of
  spin-1 antiferromagnetic {B}ose-{E}instein condensates, {\it Phys.
  Rev. A} {\bf 86}(6),  061601 (2012).

\bibitem{KS} Kohn, R., Sternberg, P., Local minimiser and singular perturbations, {\it Proc. Roy. Soc. Edinburg Sect. A} {\bf 111}, 69-84 (1989).

\bibitem{Law98}
Law, C.K., Pu, H., Bigelow, N.P., Quantum spins mixing in spinor Bose-Einstein condensates, {\it Phys. Rev. Lett.} {\bf 81}, 5257-5261 (1998).

\bibitem{Lieb2001}
 Lieb, E.H., Seiringer, R., Yngvason, J., Bosons in a trap: A rigorous derivation of the Gross-Pitaevskii energy functional., {\it Phys. Rev. A} {\bf 61}, 043602 (2001).

\bibitem{Lieb2001-2}
Lieb, E.H., Seiringer, R., Yngvason, J., A rigorous derivation of the Gross-Pitaevskii energy functional for a two-dimensional Bose gas, {\it Comm. Math. Phys.} {\bf 224},  17-31 (2001).
%
\bibitem{BaoLim2008}  Lim, F.Y., Bao, W., Numerical methods for
  computing the ground state of spin-1 Bose-Einstein condensates in a
  uniform magnetic field, {\it Phys. Rev. E}  {\bf 78}(6), 066704  (2008).

\bibitem{LinChern2011} Lin, L.,  Chern, I-L.,
A kinetic energy reduction technique and
      characterizations of the ground states of spin--1 {B}ose-{E}instein
      condensates, {\it Discrete Contin. Dyn. Syst. Ser. B}
      {\bf 19}(4), 1119-1128 (2014).

\bibitem{LinChern2013}
 Lin, L.,  Chern, I-L., Bifurcation between 2-component and 3-component ground states of spin-1 Bose-Einstein condensates in uniform magnetic fields, ArXiv:1302.0279 (2013).


\bibitem{matuszewski2009} Matuszewski, M.,  Alexander, T.J.,
   Kivshar, Y.S., Excited spin states and phase separation in spinor
  {B}ose-{E}instein condensates, {\it Phys. Rev. A} {\bf 80}(2), 
  023602 (2009).

\bibitem{matuszewski2010} Matuszewski, M., Ground states of trapped
  spin-1 condensates in magnetic field, {\it Phys. Rev. A} {\bf 82}(5),
   053630 (2010).

\bibitem{matuszewski2008} Matuszewski, M.,  Alexander, T.J.,
  Kivshar, Y.S., Spin-domain formation in antiferromagnetic
  condensates, {\it Phys. Rev. A} {\bf 78}(2), 023632 (2008).

\bibitem{M} Modica, L., The gradient theory of phase transitions and the minimal interface criterion, {\it Arch. Rational Mech. Anal.}, {\bf 98}(2), 123-142 (1987), .

\bibitem{MM1} Modica, L., Mortola, S., Un Esempio di $\Gamma^-$-Convergenza, {\it Boll. Un. Mat. Ital. B (5)} {\bf 14},  285-299 (1977).

\bibitem{MM2} Modica, L., Mortola, S., Il Limite nella $\Gamma$-convergenza di una Famiglia di Funzionali Ellittici, {\it Boll. Un. Mat. Ital. A (5)} {\bf 14}, 526-529 (1977).

\bibitem{NCK} Navarro, R., Carretero-Gonz\'alez, R., Kevrekidis, P.G., Phase separation and dynamics of two-component Bose-Einstein condensates, {\it Phy. Rev. A} {\bf 80},  023613 (2009).

\bibitem{Ohmi98}
Ohmi, T., Machida, K., Bose-Einstein condensation with internal degrees of freedom in alkali atom gases, {\it  J. Phys. Soc. Japan } {\bf 67}, 1822-1825 (1998).

\bibitem{O} Owen, N., Nonconvex variational problems with general singular perturbations, {\it Trans. Amer. Math. Soc.} {\bf 310},  393-404 (1988).

\bibitem{ORS} Owen, N., Rubinstein, J., Sternberg, P., Minimizers and gradient flows for singularly perturbed bi-stable potentials with a Dirichlet condition, {\it Proc. Roy. Soc. London Ser. A} {\bf 429} (1877),  505-532 (1990).

\bibitem{OS}  Owen, N., Sternberg, P., Nonconvex variational problems with anisotropic perturbations, {\it Nonlinear Anal.} {\bf 16},  705-719 (1991).

\bibitem{Pitaevskii61}
Pitaevskii, L.P., Vortex Lines in an imperfect Bose Gas, {\it Sov. Phys. JETP } {\bf 13}(2), 451-454 (1961).

\bibitem{Shieh} Shieh, T., From gradient theory of phase transition to a generalized minimal interface problem with a contact energy, to be appear.

\bibitem{stamper1998} Stamper-Kurn, D., Andrews, M., Chikkatur, A.,
  Inouye, S., Miesner, H.-J., Stenger, J., Ketterle, W., Optical
  confinement of a {B}ose-{E}instein condensate, {\it Phys. Rev.
  Lett.} {\bf 80}(10), 2027 (1998) .

\bibitem{stenger1998} Stenger, J., Inouye, S., Stamper-Kurn, D.,
   Miesner, H.-J., Chikkatur, A., Ketterle, W., Spin domains in
  ground-state {B}ose--{E}instein condensates, {\it Nature} {\bf 396}(6709),
   345-348 (1998).

\bibitem{S} Sternberg, P., The effect of a singular perturbation on nonconvex variational problems, {\it Arch. Rational Mech. Anal.} {\bf 101}(3),  209-260 (1988).

\bibitem{S1} Sternberg, P., Vector-valued local minimizers of non convex variational problems, {\it Rocky Mountain J. Math.} {\bf 21}(2),  799-807(1991).

\bibitem{T} Timmermans, E., Phase separation of Bose-Einstein condensates, {\it Phys. Rev. Lett.} {\bf 81}(26), 5718-5721 (1998).

\bibitem{chien2014} Wang, Y.-S., Chien, C.-S., A two-parameter
  continuation method for computing numerical solutions of spin-1
  {B}ose--{E}instein condensates, {\it J. Comput. Phys.} {\bf 256},
   198-213 (2014).

\bibitem{Zhang-Yi-You2003}  Zhang, W., Yi, S.,You, L., Mean field ground state of a spin-1 condensate in a magnetic field, {\it New J. Phys} {\bf 5}, 77.1-77.12  (2003).


 %


%

%



\end{thebibliography}

\end{document}